\documentclass{siamltex1213}
\usepackage{amssymb,amsmath,amsfonts,amssymb}
\usepackage{graphics,graphicx,color}
\usepackage{hyperref}
\usepackage{caption}
\usepackage{subcaption}
\usepackage{bmpsize}

\usepackage{mathrsfs}                    
\renewcommand{\d}{\mathrm{d}}            
\newcommand{\D}{\mathrm{D}}

\newcommand\rvec{{\bf r}}

\newcommand\nvec{{\bf n}}

\newcommand\zhat{{\bf \hat z}}
\newcommand\xhat{{\bf \hat x}}
\newcommand\yhat{{\bf \hat y}}

\newcommand\Qvec{{\bf Q}}

\newcommand\eps{{\epsilon}}

\newcommand{\Rr}{{\mathbb R}} 

\newcommand{\qs}{q_{\mathrm{s},\lambda}}
\newcommand{\tildeqs}{\tilde{q}_{\mathrm{s},\lambda}}

\newcommand{\lambdac}{\lambda_{\mathrm{c}}}
\usepackage[utf8x]{inputenc}
\usepackage[english]{babel}

\newcommand{\Vvec}{\mathbf{V}}           
\newcommand{\Ivec}{\mathbf{I}}           
\DeclareMathOperator{\tr}{tr}            
\newcommand{\abs}[1]{\left| #1 \right|}  

\title{Order Reconstruction for Nematics on Squares and Regular Polygons: A Landau-de Gennes Study}
\author{Giacomo Canevari, Apala Majumdar \& Amy Spicer}

\begin{document}
\maketitle
\begin{abstract}
We construct an order reconstruction (OR)-type Landau-de Gennes critical point on a square domain of edge length $\lambda$, motivated by the well order reconstruction solution 
numerically reported in \cite{kraljmajumdar}. The OR critical point is distinguished by an uniaxial cross with negative scalar order parameter along the square diagonals.
The OR critical point is defined in terms of a saddle-type critical point of an associated scalar variational problem. The OR-type critical point is globally stable for small~$\lambda$ 
and undergoes a supercritical pitchfork bifurcation in the associated scalar variational setting. We consider generalizations of the OR-type critical point to a regular hexagon, 
accompanied by numerical estimates of stability criteria of such critical points on both a square and a hexagon in terms of material-dependent constants.
\end{abstract}

\section{Introduction}
\label{sec:intro}
Nematic liquid crystals (LCs) are anisotropic liquids or liquids with a degree of long-range orientational order \cite{dg,virga}. 
Nematics in confinement offer ample scope for pattern formation and scientists are keen to better understand and exploit pattern formation to design new LC-based devices with advanced optical,
mechanical and even rheological properties. This paper is motivated by the well order-reconstruction solution for square wells, numerically reported in \cite{kraljmajumdar},
and its potential generalizations to other symmetric geometries. 

\vspace{.1 cm}

Nematic-filled square or rectangular wells have been widely studied in the literature \cite{tsakonas, cleaver, luo2012, lewissoftmatter}.
In \cite{tsakonas}, the authors studied the planar bistable device comprising a periodic array of micron-scale shallow nematic-filled square or rectangular wells.
The well surfaces were treated to induce tangent or planar boundary conditions so that the well molecules in contact with these surfaces are constrained to be in the plane of the surfaces.
In the absence of any external fields, the authors observe at least two different static equilibria: the diagonal state for which the molecules roughly align along one of the square diagonals
and the rotated state for which the molecules roughly rotate by $\pi$ radians between a pair of opposite edges. 

\vspace{.1 cm}

In \cite{kraljmajumdar}, the authors numerically model this device within the Landau-de Gennes (LdG) theory for nematic LCs. The LdG theory describes the nematic state by a symmetric, traceless 
$3\times 3$ matrix --- the $\Qvec$-tensor order parameter that is described in Section~\ref{sec:prelim} below. In \cite{kraljmajumdar}, the authors study static equilibria in the LdG framework, 
on a square domain with tangent boundary conditions. For square dimensions much larger than a material and temperature-dependent length scale known as the biaxial correlation length,
the authors recover the familiar diagonal and rotated solutions. For squares with edge length comparable to the biaxial correlation length, the authors find a new well order reconstruction
solution (WORS) for which the LdG $\Qvec$-tensor has a constant set of eigenvectors, one of which is $\zhat$ --- the unit vector in the $z$-direction. The WORS has a ``uniaxial'' diagonal cross
along which the LdG $\Qvec$-tensor has two equal positive eigenvalues, surrounded by a ring of ``maximal biaxiality'' for which the LdG $\Qvec$-tensor has a zero eigenvalue, matched by Dirichlet
conditions on the square edges. The WORS is interesting because it is a two-dimensional example of an order reconstruction solution on a square i.e. the LdG $\Qvec$-tensor mediates between
the diagonal cross connecting the four vertices and the Dirichlet edge conditions without any distortion of the eigenframe but by sheer variations in the eigenvalues of the LdG $\Qvec$-tensor,
referred to as \emph{eigenvalue exchange} in the literature. From an applications point of view, it can potentially offer very different optical properties to the conventional diagonal
and rotated solutions should it be experimentally realized. Further, for squares with edge length less than a certain material-dependent and temperature-dependent critical length,
the WORS appears to be unique stable LdG equilibrium, as suggested by the numerics in \cite{kraljmajumdar}.

\vspace{.1 cm}

Order reconstruction (OR) solutions have a long history in the context of nematic LCs. They were reported in \cite{sluckin, penzenstadler, gartlandmkaddem} for nematic defect cores
where the defect core is surrounded by a torus of maximal biaxiality; the torus mediates between or connects the nematic state at the defect core and the nematic states inside the torus itself.
OR solutions were successfully studied for hybrid nematic cells, typically consisting of a layer of nematic material sandwiched between a pair of parallel plates, each of which has a preferred
boundary nematic orientation \cite{kelly, bisi2003, bisi2004}. In \cite{kelly}, the authors consider the case of orthogonal preferred boundary orientations. For small cell gaps, the authors find an 
OR solution with a constant eigenframe that connects the two conflicting boundary alignments through one-dimensional eigenvalue variations across the normal to the plates. 
Indeed, the OR solution is the only observable solution (and hence globally stable) for cell gaps smaller than a certain critical value. 
For larger cell gaps, the authors observe familiar twisted profiles for which the eigenvectors rotate continuously throughout the cell to match the boundary alignments. 
The authors numerically compute a bifurcation diagram and show that the OR solution undergoes a supercritical pitchfork bifurcation to the familiar twisted solutions at a critical cell gap. 
In \cite{lamy2014}, the author rigorously studies the hybrid cell in a one-dimensional variational setting, in the LdG framework, and rigorously proves the existence of an OR solution
and the supercritical pitchfork bifurcation as the cell gap increases, at least for a range of temperatures. In \cite{bisi2003}, the authors consider the hybrid cell problem for 
non-orthogonal preferred boundary alignments. Their findings are contrasting to those of \cite{kelly} in the sense that they find an unstable OR solution for cell gaps larger
than a critical value and the familiar twisted solutions are always preferred irrespective of cell gap.

We provide a semi-analytic description of the numerically discovered WORS in this paper, and study its stability properties as a function of square size, denoted by~$\lambda$. 
We work in the LdG framework and hence, study LdG $\Qvec$-tensors on a square with edge length~$\lambda$. We impose Dirichlet tangent conditions on the edges, 
consistent with the experiments reported in \cite{tsakonas}, but there is a natural mismatch at the vertices. We truncate the vertices and hence, study a Dirichlet boundary-value problem
on a truncated square, with four long edges that are common to the original square and four short edges that are straight lines connecting the long edges. We conjecture that the artificial
short edges do not change the qualitative conclusions of our work, as is corroborated by the numerics in Section~\ref{sec:numerics}.  

We work at a fixed temperature below the nematic supercooling temperature, largely for analytic convenience. The fixed temperature only depends on material-dependent constants
and will be physically relevant for certain classes of nematic materials (see Sections~\ref{sec:prelim} and \ref{sec:reduction} below). We look for LdG equilibria which have a constant
eigenframe with $\zhat$ as an eigenvector, with a uniaxial cross along the square diagonals as described above. We parameterize these critical points by three order parameters,
$q_1$, $q_2$ and $q_3$, and the Dirichlet conditions translate into Dirichlet conditions for these variables. At the fixed temperature under consideration, 
we can prove the existence of a class of LdG critical points with $q_2 =0$ and constant $q_3$, with simply one degree of freedom labelled by $q_1$, for all values of $\lambda$.
We think of $q_1$ as a measure of the in-plane alignment of the nematic molecules. These critical points have a constant eigenframe by construction and the uniaxial cross is equivalent
to $q_1=0$ along the square diagonals. Hence, our task is reduced to constructing LdG critical points for which $q_1=0$ along the square diagonals. We appeal to ideas from saddle solutions
for bistable Allen-Cahn equations in \cite{fife, schatzman} to interpret $q_1$ as a minimizer of a scalar variational problem on a quadrant of a truncated square with Dirichlet conditions.
We then define $q_1$ on the truncated square by an odd reflection of the quadrant solution across the diagonals, yielding a LdG critical point on the entire domain defined in terms
of a minimizer of a scalar problem on a quadrant of the square. This critical point has $q_1=0$ along the diagonals by construction, exists for all $\lambda$ and reproduces
the qualitative properties of the WORS. We refer to this LdG critical point as being the \emph{OR} or \emph{saddle-type} LdG critical point in the rest of the paper.

The OR LdG critical point is the unique critical point (and hence globally stable) for small $\lambda$, as can be demonstrated by a uniqueness argument used in \cite{lamy2014}.
Further, in Section~\ref{sec:ORinstability}, we prove that the OR LdG critical point, which can also be interpreted as a critical point of a scalar variational problem, 
undergoes a supercritical pitchfork bifurcation in the scalar setting as $\lambda$ increases. In other words, it is unstable for large $\lambda$ and hence, not observed 
for large micron-scale wells studied in \cite{tsakonas, luo2012}. This bifurcation result can be viewed as a non-trivial two-dimensional generalization of the one-dimensional bifurcation 
result for a hybrid cell in \cite{lamy2014}. The strategy is the same --- we appeal to Crandall-Rabinowitz theorem but we have partial differential equations and not ordinary differential equations, 
which introduce technical difficulties. Further, it is not a priori obvious that the pitchfork bifurcation will carry over to a two-dimensional problem. However, there are now detailed bifurcation
plots for solutions in the full LdG framework \cite{robinson}, which show that the supercritical pitchfork bifurcation can also be observed in the full LdG setting,
and not merely the  reduced scalar setting.

In Section~\ref{sec:numerics}, we study the gradient flow model for the LdG energy on a square with Dirichlet conditions and special OR-type initial conditions. 
As expected, the long-time dynamics converges to the WORS for small $\lambda$ and we numerically compute estimates for the critical $\lambda$. In \cite{kraljmajumdar}, 
the authors solve the LdG Euler-Lagrange equations with effectively constant initial conditions and hence, our numerical experiments are not identical to those in \cite{kraljmajumdar}. 
The critical $\lambda$ is proportional to the biaxial correlation length at our choice of the fixed temperature, as expected from the numerics in \cite{kraljmajumdar}. 
Further, we numerically reproduce the supercritical pitchfork bifurcation, as the WORS solution loses stability for larger values of $\lambda$. In Section~\ref{sec:hexagon}, 
we prove the existence of an OR-type solution on a regular hexagon of edge length $\lambda$, for all values of $\lambda$ at the fixed temperature. 
The method of proof is entirely different to that of a square; we appeal to Palais' principle of symmetric criticality \cite{palais, lamy2014}. 
We numerically solve the LdG gradient flow model on a regular hexagon and again find the OR-type solution, featured by a ring of maximal biaxiality centered at the centre of the hexagon,
for small $\lambda$. The critical stability criterion is again proportional to the biaxial correlation length. This suggests that OR-type solutions may be generic
for regular convex polygons with an even number of sides, raising many interesting questions about the interplay of geometry, symmetry, temperature and multiplicity of LdG equilibria.  

\section{Preliminaries}
\label{sec:prelim}

We model nematic profiles on two-dimensional prototype geometries within the Landau-de Gennes (LdG) theoretical framework.
The LdG theory is one of the most powerful continuum theories for nematic liquid crystals and describes the nematic state by a macroscopic order parameter 
--- the LdG $\mathbf{Q}$-tensor that is a macroscopic measure of material ansiotropy. 
The LdG $\Qvec$-tensor  is a symmetric traceless $3\times 3$ matrix i.e. 
$\Qvec \in S_0 := \left\{ \Qvec\in \mathbb{M}^{3\times 3}\colon Q_{ij} = Q_{ji}, \ Q_{ii} = 0 \right\}$.
A $\Qvec$-tensor is said to be (i) isotropic if~$\Qvec=0$, (ii) uniaxial if $\Qvec$
has a pair of degenerate non-zero eigenvalues and (iii) biaxial if~$\Qvec$ has three distinct eigenvalues~\cite{dg,newtonmottram}. A
uniaxial $\Qvec$-tensor can be written as $\Qvec_u = s \left(\nvec \otimes \nvec - \Ivec/3\right)$ with~$\Ivec$ the $3\times 3$ identity matrix,
$s\in\Rr$ and~$\nvec\in S^2$, a unit vector. The scalar, $s$, is an order parameter which measures the degree of orientational order.
The vector, $\nvec$, is referred to as the ``director'' and labels the single distinguished direction of uniaxial nematic alignment~\cite{virga,dg}. 


We work with a simple form of the LdG energy given by
\begin{equation} \label{eq:2} 
 I[\Qvec] :=  \int_{\Omega} \frac{L}{2} \left|\nabla\Qvec \right|^2 + f_B(\Qvec) \, \mathrm{d}A
\end{equation}

where $\Omega\subseteq\Rr^2$ is a two-dimensional domain,
\begin{equation} \label{eq:3}
 |\nabla \Qvec |^2 := \frac{\partial Q_{ij}}{\partial r_k}\frac{\partial Q_{ij}}{\partial r_k},
 \qquad f_B(\Qvec) := \frac{A}{2} \tr\Qvec^2 - \frac{B}{3} \tr\Qvec^3 + \frac{C}{4}\left(\tr\Qvec^2 \right)^2.
\end{equation}

The variable~$A = \alpha (T - T^*)$ is the re-scaled temperature, $\alpha$, $L$, $B$, $C>0$ are material-dependent constants
and~$T^*$ is the characteristic nematic supercooling temperature~\cite{dg,newtonmottram}.
Further~$\rvec:=(x, \, y)$, $\tr\Qvec^2 = Q_{ij}Q_{ij}$ and $\tr\Qvec^3 = Q_{ij} Q_{jk}Q_{ki}$ for $i$, $j$, $k = 1, \, 2, \, 3$.
It is well-known that all stationary points of the thermotropic
potential, $f_B$, are either uniaxial or isotropic~\cite{dg,newtonmottram,ejam2010}.
The re-scaled temperature~$A$ has three characteristic values: (i)~$A=0$, below which the isotropic phase $\Qvec=0$ loses stability, 
(ii) the nematic-isotropic transition temperature, $A={B^2}/{27 C}$, at which $f_B$ is minimized by the isotropic
phase and a continuum of uniaxial states with $s=s_+ ={B}/{3C}$ and
$\nvec$ arbitrary, and (iii) the nematic supercooling temperature, $A = {B^2}/{24 C}$, above which the ordered nematic equilibria do not exist.

We work with $A<0$ i.e. low temperatures and a large part of the paper focuses on a special temperature,
$A=-{B^2}/{3C}$. This temperature is not special except for analytical convenience, as will be exemplified in the following sections. 
Some of our results can be readily generalized to all temperatures, $A<0$. For a given $A<0$, let
$\mathscr{N} := \left\{ \Qvec \in S_0\colon \Qvec = s_+ \left(\nvec\otimes \nvec - \Ivec/3 \right) \right\}$ 
denote the set of minimizers of the bulk potential, $f_B$, with
\[
 s_+ := \frac{B + \sqrt{B^2 + 24|A| C}}{4C}
\]
and~$\nvec \in S^2$ arbitrary. In particular, this set is relevant to our choice of Dirichlet conditions for boundary-value problems.

We non-dimensionalize the system using a change of variables, $\bar{\rvec} = \rvec/ \lambda$,
where $\lambda$ is a characteristic length scale of the system.
The re-scaled LdG energy functional is then given by
\begin{equation} \label{eq:rescaled}
 \overline{I}[\Qvec] := \frac{I[\Qvec]}{L \lambda} = \int_{\overline{\Omega}}\frac{1}{2}\left| \overline{\nabla} \Qvec \right|^2 
 + \frac{\lambda^2}{L} f_B\left(\Qvec \right) \, \overline{\mathrm{d}A}.
\end{equation}
In~\eqref{eq:rescaled}, $\overline{\Omega}$ is the re-scaled domain, $\overline{\nabla}$ is the gradient with respect to
the re-scaled spatial coordinates and $\overline{\mathrm{d}A}$ is the re-scaled area element. The associated Euler-Lagrange equations are 
%
\begin{equation} \label{eq:6} 
 \bar{\Delta} \Qvec = \frac{\lambda^2}{L} \left\{ A\Qvec + B\left(\Qvec\Qvec - \frac{\Ivec}{3}|\Qvec|^2 \right)
 - C|\Qvec|^2 \Qvec \right\},
\end{equation}
where $(QQ)_{ik} = Q_{ij}Q_{jk}$ with $i$, $j$, $k=1, \, 2, \, 3$
and $\Ivec$ is the $3\times 3$ identity matrix. The system~\eqref{eq:6} comprises five coupled nonlinear
elliptic partial differential equations.
In what follows, we analytically and numerically study solutions of~\eqref{eq:6} on prototype two-dimensional geometries
and solution stability as a function of the length, $\lambda$, which is a measure of the size of the domain. 
We treat $A$, $B$, $C$, $L$ as fixed constants and vary $\lambda$; the analytical results are asymptotic in nature 
but are validated by numerical simulations for $\lambda \in (0.5\times 10^{-6}, \, 0.5\times 10^{-4})~m$.
The LdG theory is believed to be valid for such length scales and our analysis is hence, 
corroborated by numerical simulations for physically relevant length scales as stated above.
In what follows, we drop the \emph{bars} and all statements are to be understood in terms of the re-scaled variables.


\section{A Scalar Variational Problem for \texorpdfstring{$A=-{B^2}/{3C}$}{A = -B2/3C}}
\label{sec:reduction}

We take ~$\Omega\subseteq\Rr^2$ to be a truncated unit square, whose diagonals lie along the axes:
\begin{equation} \label{eq:d3}
 \Omega := \left\{(x, \, y)\in\Rr^2\colon |x| < 1 - \varepsilon, \ |y| < 1 - \varepsilon, \ |x+y| < 1, \ |x-y| < 1 \right\}.
\end{equation}
The boundary, $\partial\Omega$, consists of four ``long'' edges~$C_1, \, \ldots, \, C_4$,  parallel to the lines ~$y = x$ and~$y = -x$,
and four ``short'' edges~$S_1, \, \ldots, \, S_4$, of length~$2\varepsilon$, parallel to the $x$ and $y$-axes respectively.
The four long edges~$C_i$ are labeled counterclockwise and $C_1$ is the edge contained in the first quadrant, i.e.
\[
 C_1 := \left\{(x, \, y)\in\Rr^2\colon x + y = 1, \ \varepsilon \leq x \leq 1 - \varepsilon \right\}.
\]
The short edges~$S_i$ are introduced to remove the sharp square vertices. They are also labeled counterclockwise and
\[
 S_1 := \left\{(1 - \varepsilon, \, y)\in\Rr^2\colon |y|\leq \varepsilon\right\}.
\]

We work with Dirichlet conditions on $\partial \Omega$. Following the existing literature on planar multistable
nematic systems~\cite{tsakonas, luo2012, kraljmajumdar}, we impose \emph{tangent} uniaxial Dirichlet conditions on the long edges, $C_1, \, \ldots, \, C_4$.
These tangent conditions simply require the uniaxial director to be tangent to the long edges and we fix $\Qvec = \Qvec_{\mathrm{b}}$ on $C_1, \, \ldots, \, C_4$ where
\begin{equation} \label{eq:bc1}
 \Qvec_{\mathrm{b}}(\rvec) := \begin{cases} 
  s_+\left( \nvec_1 \otimes \nvec_1 - \dfrac{\Ivec}{3} \right)  & \textrm{for } \rvec \in C_1 \cup C_3 \\
  s_+\left( \nvec_2 \otimes \nvec_2 - \dfrac{\Ivec}{3} \right)  & \textrm{for } \rvec \in C_2 \cup C_4;
 \end{cases}
\end{equation}
and
\[
 \nvec_1 := \frac{1}{\sqrt{2}}\left(-1, \, 1\right), \qquad  \nvec_2 := \frac{1}{\sqrt{2}}\left(1, \, 1 \right).
\]
We note that $\Qvec_{\mathrm{b}}\in\mathscr{N}$ on $C_1, \ldots C_4$ and the choice of $\nvec_1$ and $\nvec_2$ is dictated by the tangent boundary condition.

We prescribe Dirichlet conditions on the short edges too; these conditions are somewhat artifical and used purely for mathematical convenience.
However, some of our analytical results also hold for Neumann conditions on the short edges and these free boundary conditions are physically relevant.
Further, the numerical simulations in Section~\ref{sec:numerics} only use the tangent Dirichlet conditions in~\eqref{eq:bc1}
and do not employ the artificial Dirichlet conditions on the short edges and yet, the numerical results are consistent with the analysis
of our Dirichlet boundary-value problem. We believe that our choice of the Dirichlet conditions on the short edges, although artifical,
provides a nice platform for mathematical analysis and the analysis provides useful insight into more realistic boundary-value problems too.

The Dirichlet condition on the short edges is defined in terms of a function
\begin{equation} \label{eq:g}
  g(s) := \frac{s_+}{2} \left(e^{-\mu\varepsilon}\frac{e^{\mu s} - e^{-\mu s}}{e^{\mu\varepsilon} - e^{-\mu\varepsilon}}
  - e^{-\mu s} + 1\right) \quad \textrm{for } 0 < s < \varepsilon; 
  \qquad \mu := \frac{\lambda B}{(CL)^{1/2}}
\end{equation}
and we take $g(s) = s_+/2$ for $s>\eps$ and $g(s ) = - g( - s)$ for $s<0$.
We fix $\Qvec = \Qvec_{\mathrm{b}}$ on $S_1, \, \ldots, \, S_4$ where
\begin{equation} \label{eq:bc2}
 \Qvec_{\mathrm{b}} := \begin{cases}
  g(y) \left(\nvec_1 \otimes \nvec_1 - \nvec_2\otimes \nvec_2 \right) - \dfrac{s_+}{6}\left(2 \hat{\mathbf{z}}\otimes\hat{\mathbf{z}}
  - \nvec_1\otimes \nvec_1 - \nvec_2\otimes \nvec_2 \right)  & \textrm{on  } S_1\cup S_3, \\
  g(x)\left(\nvec_1 \otimes \nvec_1 - \nvec_2\otimes \nvec_2 \right) - \dfrac{s_+}{6}\left(2 \hat{\mathbf{z}}\otimes\hat{\mathbf{z}}
  - \nvec_1\otimes \nvec_1 - \nvec_2\otimes \nvec_2 \right) & \textrm{on  } S_2\cup S_4.
 \end{cases}
\end{equation}
This Dirichlet condition is artificial for two reasons: (i) $\Qvec_{\mathrm{b}} \notin \mathscr{N}$ on $S_1, \, \ldots, \, S_4$
i.e. $\Qvec_{\mathrm{b}}$ is biaxial on these edges and (ii) $\Qvec_{\mathrm{b}}$ is not tangent on these edges.
However, these edges are short by construction and we conjecture that the qualitative solution trends are
not affected by the choice of~$\Qvec_{\mathrm{b}}$ on $S_1, \, \ldots, \, S_4$. Given the Dirichlet conditions~\eqref{eq:bc1} and~\eqref{eq:bc2},
we define our admissible space to be
\begin{equation} \label{eq:admissible}
 \mathscr{A} := \left\{ \Qvec \in W^{1,2}\left(\Omega, \, S_0 \right)\!\colon \Qvec = \Qvec_{\mathrm{b}}~\textrm{on} ~\partial \Omega \right\}.
\end{equation}
It is straightforward to prove the existence of a global minimizer of the re-scaled functional~\eqref{eq:rescaled} in the admissible space~$\mathscr{A}$,
for all $A<0$ and for all values of $\lambda>0$.
 
In~\cite{kraljmajumdar}, the authors numerically find the order reconstruction solution for nano-scale wells or equivalently, 
for small~$\lambda$ in our framework. Our work is motivated by an analytic characterization of the order reconstruction (OR) solution
reported in~\cite{kraljmajumdar}; the OR solution is a critical point of~\eqref{eq:rescaled} with two key properties:
(i) the corresponding $\Qvec$-tensor has a constant eigenframe i.e. three constant eigenvectors, one of which is~$\hat{\mathbf{z}}$ 
--- the unit vector in the $z$ direction, (ii) this critical point is distinguished by an uniaxial cross with negative scalar order parameter along the square diagonals.

In the spirit of the numerical results reported in \cite{kraljmajumdar}, we look for critical points of the re-scaled functional~\eqref{eq:rescaled} of the form
\begin{equation} \label{eq:d1}
 \begin{split}
  \Qvec(x, \, y) &= q_1(x, \, y) \left(\nvec_1 \otimes \nvec_1 - \nvec_2\otimes \nvec_2 \right) 
  + q_2(x, \, y) \left(\nvec_1 \otimes \nvec_2 + \nvec_1\otimes \nvec_2 \right) \\
  &\qquad\qquad + q_3(x, \, y) \left(2 \hat{\mathbf{z}}\otimes\hat{\mathbf{z}} - \nvec_1\otimes \nvec_1 - \nvec_2\otimes \nvec_2 \right)
 \end{split}
\end{equation}
subject to the boundary conditions
\begin{equation} \label{eq:d2}
 q_1(x, \, y) = q_{\mathrm{b}} (x, \, y) :=
 \begin{cases} 
  -{s_+}/{2}  & \textrm{on  } C_1\cup C_3 \\
  {s_+}/{2}   & \textrm{on  } C_2\cup C_4 \\
  g(y)        & \textrm{on  } S_1\cup S_3 \\
  g(x)        & \textrm{on  } S_2\cup S_4;
 \end{cases}
\end{equation}
$q_2 = 0$ and $q_3 = - {s_+}/{6}$ on~$\partial\Omega$.
Critical points of the form~\eqref{eq:d1} mimic the order reconstruction solution if $q_2 = 0$, 
which ensures a constant eigenframe with $\hat{\mathbf{z}}$ being an eigenvector, 
and~$q_1$ vanishes along the square diagonals (the coordinate axes in our setting), so that 
$\Qvec = q_3 (x, y) \left(3\hat{\mathbf{z}}\otimes\hat{\mathbf{z}} - {\Ivec}\right)$ on $x=0$ and $y=0$.
We first present an elementary result regarding the existence of such critical points. 

\begin{proposition} \label{prop:1}
 The LdG Euler-Lagrange equations~\eqref{eq:6} admit a solution of the \linebreak form~\eqref{eq:d1} on the truncated square,
 $\Omega$ defined in~\eqref{eq:d3} subject to the Dirichlet conditions~\eqref{eq:bc1} and~\eqref{eq:bc2},
 provided the functions $q_1$, $q_2$, $q_3$ satisfy the following system
 \begin{equation} \label{eq:d4}
  \begin{aligned}
   \Delta q_1 &= \frac{\lambda^2}{L}\left\{A q_1 + 2B q_1 q_3 + C\left( 2q_1^2 + 2q_2^2 + 6q_3^2 \right) q_1 \right\} \\
   \Delta q_2 &= \frac{\lambda^2}{L}\left\{A q_2 + 2B q_2 q_3 + C\left( 2q_1^2 + 2q_2^2 + 6q_3^2 \right) q_2 \right\} \\
   \Delta q_3 &= \frac{\lambda^2}{L}\left\{A q_3 + B \left( \frac{1}{3}\left( q_1^2 + q_2^2 \right) - q_3^2 \right) + C\left( 2q_1^2 + 2q_2^2 + 6q_3^2 \right) q_3\right\} 
  \end{aligned}
 \end{equation}
 and the boundary conditions in~\eqref{eq:d2}.
\end{proposition}

\begin{proof}
 Consider the energy functional $J[q_1, \, q_2, \, q_3]$ defined below:
 \begin{equation} \label{eq:d5}
  \begin{split}
   &J[q_1, q_2, q_3]: = \int_{\Omega} \left( |\nabla q_1|^2 + |\nabla q_2|^2 + 3 |\nabla q_3|^2 \right) \, \d A  \\
   &\quad + \int_{\Omega} \frac{\lambda^2}{L}\left( A \left( q_1^2 + q_2^2 + 3 q_3^2 \right) + C\left( q_1^2 + q_2^2 + 3 q_3^2 \right)^2 
   + 2 B q_3 \left( q_1^2 + q_2^2\right) - 2 B q_3^3 \right) \, \d A.
  \end{split}
 \end{equation}
 We can prove the existence of a global minimizer of the functional~$J$ in~\eqref{eq:d5}
 among the triplets $(q_1, \, q_2, \, q_3)\in W^{1,2}(\Omega, \, \Rr^3)$ satisfying the boundary conditions~\eqref{eq:d2}
 from the direct methods in the calculus of variations, since $J$ is both coercive and weakly lower semi-continuous~\cite{evans}.
 The system of elliptic partial differential equations in~\eqref{eq:d4} comprise the Euler-Lagrange equations associated with~$J$
 in~\eqref{eq:d5} and the globally minimizing $(q_1, \, q_2, \, q_3)$ are classical solutions of the system~\eqref{eq:d4}.
 Once we obtain the solutions of the system~\eqref{eq:d4}, we can construct the corresponding tensor in~\eqref{eq:d1}
 and check that it is an exact solution of the LdG Euler-Lagrange equations in~\eqref{eq:6} by direct substitution.
\end{proof}

We do not have results on the multiplicity of solutions of the system~\eqref{eq:d4} for arbitrary~$\lambda$
but it is straightforward to check that there is a branch of solutions, $(q_1, \, 0, \, q_3)$, of the system~\eqref{eq:d4},
for all~$\lambda>0$ and for all~$A <0$. This solution branch does have a constant eigenframe but we need stronger properties
to analyze the order reconstruction solution. We henceforth, restrict ourselves to a special temperature
\begin{equation} \label{eq:d6}
 A = -\frac{B^2}{3C}
\end{equation}
for which $s_+ = {B}/{C}$. This temperature is special because the system~\eqref{eq:d4} admits a branch of solutions,
$(q_1, \, q_2, \,  q_3) = (q(x, \, y), \, 0, \, -{B}/{6C})$ at this temperature, consistent with the Dirichlet conditions in~\eqref{eq:d2},
for all~$\lambda>0$. It is simpler to analyze a branch of solutions with one variable $q(x, \, y)$ than solution branches with multiple variables and hence,
we restrict ourselves to this temperature and this solution branch in the remainder of this paper.

\begin{proposition} \label{prop:2}
 For $A = - {B^2}/{3C}$ and for all $\lambda>0$, there exists a branch of solutions of the system~\eqref{eq:d4} given by
 \begin{equation} \label{eq:d7}
  (q_1, \, 0, \, q_3) = \left(q_{\mathrm{min}}(x,y), \, 0, \, -\frac{B}{6C} \right)
 \end{equation}
 consistent with the Dirichlet conditions in~\ref{eq:d2}.
 This branch is defined by a minimizer, $q_{\mathrm{min}}$, of the following energy:
 \begin{equation} \label{eq:d8}
  H[q]: = \int_{\Omega} |\nabla q|^2 + \frac{\lambda^2}{L}\left(C q^4 - \frac{B^2}{2C} q^2\right) \, \d A
 \end{equation}
 subject to~\eqref{eq:d2} and is hence, a classical solution of
 \begin{equation} \label{eq:d9}
  \Delta q =  \frac{\lambda^2}{L}\left(2C q^3 - \frac{B^2}{2C} q \right),
 \end{equation}
 which is precisely the first equation in~\eqref{eq:d4} with $q_2 = 0$ and $q_3=-{B}/{6C}$. 
 We have the bounds
 \begin{equation} \label{eq:d10}
 -\frac{B}{2C} \leq q_{\mathrm{min}} \leq \frac{B}{2C}.
 \end{equation}
\end{proposition}

\begin{proof}
 We can check that the solution branch defined by~\eqref{eq:d7} is a solution of the system~\eqref{eq:d4} 
 at~$A = - {B^2}/{3C}$ for all $\lambda>0$, if the function~$q$ is a solution of the partial differential equation~\eqref{eq:d9}
 subject to the Dirichlet conditions~\eqref{eq:d2}.
 
 Let~$q_{\mathrm{min}}$ be a minimizer of the functional~$H$ defined in~\eqref{eq:d8}, in the admissible space,
 $\mathscr{A}_q := \left\{ q\in W^{1,2}(\Omega) \colon q \textrm{ satisfies \eqref{eq:d2} on } \partial \Omega \right\}$.
 The existence of a minimizer follows from the direct methods in the calculus of variations.
 Then~$q_{\mathrm{min}}$ is a classical solution of the associated Euler-Lagrange equation~\eqref{eq:d9} which ensures that the triplet 
 $(q_1, \, q_2, \, q_3) = (q_{\mathrm{min}}, \, 0, \, -{B}/{6C})$ is a solution of the system~\eqref{eq:d4} yielding a critical point of 
 the LdG Euler-Lagrange equations in~\eqref{eq:6} via the representation~\eqref{eq:d1}.
 The bounds~\eqref{eq:d10} are a straightforward consequence of the maximum principle and the Dirichlet conditions in~\eqref{eq:d2}.
\end{proof}

\begin{lemma} \label{lem:1}
 There exists a number~$\lambda_0 > 0$ such that, for any~$\lambda < \lambda_0$,
 the solution branch defined by $(q_1, \, q_2, \, q_3) = (q_{\mathrm{min}}, \,  0, \, -{B}/{6C})$  in Proposition~\eqref{prop:2}
 is the unique critical point of the LdG energy \eqref{eq:rescaled}. 
\end{lemma}
\begin{proof}
 This is an immediate consequence of a general uniqueness result for critical points of the energy~\eqref{eq:rescaled} in Lemma~8.2 of \cite{lamy2014}.
 A critical point, $\Qvec_{\mathrm{c}}$, of the LdG Euler-Lagrange equations in~\eqref{eq:6} is bounded as an immediate consequence of the maximum principle
 (see \cite{amaz, ejam2010}) i.e. $\left| \Qvec_{\mathrm{c}} \right| \leq M (A, \, B, \, C)$ and the bound $M$ is independent of ${\lambda^2}/{L}$.
 The key step is to note that the LdG energy is strictly convex on the set 
 \[
  \left\{\Qvec \in W^{1,2}(\Omega, \,  S_0)\colon  \left| \Qvec \right| \leq M \right\}
 \]
 for sufficiently small~$\lambda$ i.e. for ${\lambda^2}/{L} < \lambda_1(\Omega, \, A, \, B, \, C)$
 where the constant $\lambda_1$ depends on the domain, temperature and material constants. 
 Hence, the LdG energy~\ref{eq:rescaled} has a unique critical point in this regime.
 
 Let $A = - {B^2}/{3C}$; then the triplet $(q_1, \, q_2, \, q_3) = (q_{\mathrm{min}}, 0, -{B}/{6C})$ 
 in Proposition~\ref{prop:2} defines a LdG critical point of the form~\eqref{eq:d1} for all~$\lambda>0$.
 From the strict convexity of the LdG energy on the set of bounded $\Qvec$-tensors for small~$\lambda$ and fixed $L > 0$,
 we deduce that this must be the unique LdG critical point, and hence the unique global LdG energy minimizer for sufficiently small~$\lambda$.
 This yields the desired conclusion.
\end{proof}  

\begin{lemma} \label{lem:2}
 The function $q_{\mathrm{min}}\colon\Omega\to\Rr$ defined in Proposition~\ref{prop:2} vanishes along the square diagonals defined by $x=0$ and $y=0$,
 provided that~$\lambda < \lambda_0$ given by Lemma~\ref{lem:1}.
\end{lemma}
\begin{proof} 
 We make the elementary observation that if $q(x, \, y)$ is a solution of~\eqref{eq:d9} subject to~\eqref{eq:d2},
 then so are the functions $q(-x, \, -y)$, $- q(-x, \, y)$, $-q(x, \, -y)$. 
 We combine this symmetry result with the uniqueness result for small~$\lambda$ in Lemma~\ref{lem:1} above (also see~\cite{lamy2014})
 to get the desired conclusion in the $\lambda\to 0$ limit, for example, simply use $q(x, \, y) = - q(-x, \, y)$ with $x=0$
 to deduce that $q(0, \, y) = 0$ and we can use an analogous argument to show that $q=0$ along $y=0$.
\end{proof}

From Lemmas~\ref{lem:1} and~\ref{lem:2}, we deduce that there is a unique LdG critical point of the form
\begin{equation}\label{eq:s1}
 \Qvec_{\mathrm{min}}(x, \, y) = q_{\mathrm{min}}(x, \, y)\left(\nvec_1 \otimes \nvec_1 - \nvec_2\otimes \nvec_2 \right) 
 - \frac{B}{6C}\left(2 \hat{\mathbf{z}}\otimes\hat{\mathbf{z}} - \nvec_1\otimes \nvec_1 - \nvec_2\otimes \nvec_2 \right)
\end{equation}
for sufficiently small $\lambda$, where $q_{\mathrm{min}}$ is a global minimizer of the functional~$H$ in Proposition~\ref{prop:2} 
such that $q_{\mathrm{min}} = $ on $x=0$ and $y=0$. This critical point has a constant eigenframe and has a uniaxial cross of negative
scalar order parameter along the square diagonals (the coordinate axes) and hence, has all the qualitative properties of the order reconstruction
solution reported in~\cite{kraljmajumdar}. However, global minimizers of~$H$ need not satisfy the symmetry property,
$q_{\mathrm{min}}=0$ on the coordinate axes, for large~$\lambda$. 
This will be demonstrated by the following proposition, which characterize the asymptotic behaviour of~$q_{\min}$ as~$\lambda\to+\infty$.
We introduce the following notation: for any set~$E\subseteq\Rr^2$, we define the~$\Omega$-perimeter of~$E$ as
\[
 \mathrm{Per}_\Omega(E) := \sup\left\{\int_E \mathrm{div} \, \varphi \,\d A \colon
 \varphi\in C^1_{\mathrm{c}}(\Omega), \ |\varphi| \leq 1 \textrm{ on } \Omega\right\}.
\]
If~$E$ has a smooth boundary, then the Gauss-Green formula implies that~$\mathrm{Per}_\Omega(E) = \mathrm{length} \, (\partial E \cap\Omega)$.
We denote by~$\mathscr{B}$ the class of functions~$q$, defined on~$\Omega$, that only take the values~$B/2C$, $-B/2C$ and are such that
$\mathrm{Per}_\Omega\{q = -B/2C\} < +\infty$.
For any~$\lambda > 0$, we let~$q_{\min, \lambda}$ be a minimizer of~$H$.

\begin{proposition} \label{prop:3}
 There exists a subsequence~$\lambda_j\nearrow +\infty$ and a function~$q_\infty\in\mathscr{B}$ such that
 $q_{\min, \lambda_j}\to q_\infty$ in~$L^1(\Omega)$ and a.e. Moreover,~$q_\infty$ is a minimizer of the 
 functional~$J\colon L^1(\Omega)\to (-\infty, \, +\infty]$ given by
 \begin{equation} \label{J}
  J[q] := k \, \mathrm{Per}_\Omega \left\{ q = -\frac{B}{2C}\right\} 
  + \int_{\partial\Omega} \phi(q_{\mathrm{b}}(\rvec), \, q(\rvec)) \, \d \mathrm{s}
 \end{equation}
 if~$q\in\mathscr{B}$, and by~$J[q] := +\infty$ otherwise.
 Here~$q_{\mathrm{b}}$ is the boundary datum defined by~\eqref{eq:d2} and
 \begin{gather}
  \phi(s, \, t) := 2\sqrt{\frac{C}{L}} \abs{\int_s^t \left(\frac{B^2}{4C^2} - \tau^2\right) \, \d \tau} = 
  2\sqrt{\frac{C}{L}} \abs{\frac{1}{3}(s^3 - t^3) - \frac{B^2}{4C^2}(s - t)}, \label{J:phi} \\
  k := \phi\left(-\frac{B}{2C}, \, \frac{B}{2C}\right) = \frac{B^3}{3C^3} \sqrt{\frac{C}{L}}.\label{J:k}
 \end{gather}
 In Equation~\eqref{J}, the value~$q(\rvec)$ for~$\rvec\in\partial\Omega$ is understood as the inner trace of~$q$ at the point~$\rvec$.
\end{proposition}

The proof of this result follows along the lines of the analysis carried out by Modica-Mortola~\cite{ModicaMortola} and by Sternberg~\cite{Sternberg}.
In particular, Proposition~\ref{prop:3} is a direct consequence of~\cite[Theorem~7.10]{Braides}, combined with e.g. \cite[Theorem~7.3 or~7.11]{Braides}.


We comment on the implications of Proposition~\ref{prop:3}.
Suppose that, for any~$\lambda>0$, the minimizer~$q_{\min, \lambda}$ mimic the OR solution,
i.e.~$q_{\min, \lambda}(x, \, y) = 0$ on the coordinate axes~$x = 0$, $y=0$, $q_{\min, \lambda} > 0$ on the first an third quadrant,
and~$q_{\min, \lambda} < 0$ on the second and fourth quadrant.
Then, the limit function~$q_\infty$ given by Proposition~\ref{prop:3} would be
\[
 q_\infty(x, \, y) = \begin{cases}
                      B/2C  & \textrm{if } xy > 0 \\
                      -B/2C & \textrm{if } xy < 0,
                     \end{cases}
\]
with sharp transition localized on the coordinate axes~$x=0$ or~$y=0$, therefore
\begin{equation} \label{energy_q_infty}
 J[q_\infty] \geq k \mathrm{Per}_\Omega \left\{q_\infty = -\frac{B}{2C} \right\} = 4k (1 - \varepsilon).
\end{equation}
We consider now the constant function~$q = B/2C$, which has no transition layer in the interior of~$\Omega$
but do not match the Dirichlet boundary condition~\eqref{eq:d2}. Nevertheless, it is an admissible comparison function
for the minimization problem associated with the functional~\eqref{J}, for which no boundary condition is imposed. 
Then, using the symmetry of the problem, the boundary condition~\eqref{eq:d2} and~\eqref{J:phi}--\eqref{J:k}, we have
\begin{equation} \label{energy_constant}
 \begin{split}
  J[B/2C] &= \int_{\partial\Omega} \phi\left(-\frac{B}{2C}, \, \frac{B}{2C}\right) \, \d\mathrm{s} \\
  &= k \, \mathrm{Length}\,(C_2 \cup C_4) + 4 \int_{-\varepsilon}^\varepsilon \phi\left(g(s), \, \frac{B}{2C}\right) \,\d\mathrm{s} \\
  &\leq 2\sqrt{2}k(1 - \varepsilon) + 8k\varepsilon.
 \end{split}
\end{equation}
Now, if we take~$\varepsilon$ sufficiently small, Equations~\eqref{energy_q_infty} and~\eqref{energy_constant} 
imply that~$J[q_\infty] > J[B/2C]$, thus contradicting the minimality of~$q_\infty$ stated by Proposition~\ref{prop:3}.
We conclude that, for large~$\lambda$, the minimizers~$q_{\min, \lambda}$ of~$H$ do not vanish on the coordinate axes.
As a consequence, LdG critical points of the form~\eqref{eq:s1} mimic the OR solution for small~$\lambda$ but not for large~$\lambda$.

In the next sections, we address the following questions: (i) does the OR solution exist for all $\lambda$ and if so,
can we provide a semi-analytic description as in~\eqref{eq:s1} with a different interpretation of~$q$ as a critical point (not a minimizer)
of the functional~$H$ in Proposition~\ref{prop:2} and (ii) how does the stability of the OR solution depend on the square size denoted by~$\lambda$.

\section{The Order Reconstruction Solution}
\label{sec:OR}

This section is devoted to an analytic definition of the OR solution reported in~\cite{kraljmajumdar} and an analysis of its qualitative properties.
In light of the numerical results in~\cite{kraljmajumdar}, we construct OR critical points of the form~\eqref{eq:s1}, such that
$q=0$ on $x=0$ and $y=0$. This necessarily implies that the corresponding $\Qvec$-tensor is uniaxial with negative order parameter
(see~\eqref{eq:s1}) on the coordinate axes. We define the corresponding $q$'s in terms of a critical point, $q_{\mathrm{s}}$ of the functional~$H$ 
in Proposition~\ref{prop:2} and our definition of $q_{\mathrm{s}}$ is analogous to saddle solutions of the bistable Allen-Cahn equation studied in~\cite{fife, schatzman}. 

We consider the Euler-Lagrange equations associated with $H$ in Proposition~\ref{prop:2}
\begin{equation} \label{AC} \tag{AC}
 \begin{cases}
  -\Delta q + \dfrac{\lambda^2}{L} f(q) = 0 & \textrm{on } \Omega \\
  q = q_{\mathrm{b}} & \textrm{on } \partial \Omega.
 \end{cases}
\end{equation}
Here, $f$ is given by
\begin{equation} \label{f}
 f(q) := 2C q^3 - \frac{B^2}{2C} q,
\end{equation}
so the equation~\eqref{AC} is of the Allen-Cahn type and $q_{\mathrm{b}}$ is defined in \eqref{eq:d2}. 

For a fixed~$\lambda>0$, we define an \emph{order reconstruction (OR) solution}, or \emph{saddle solution} to be
a classical solution~$q_{\mathrm{s}}\in C^2(\Omega)\cap C(\overline{\Omega})$ of Problem~\eqref{AC} that satisfies the sign condition
\begin{equation} \label{sign_condition}
 x y \, q_{\mathrm{s}}(x, \, y)\geq 0 \qquad \textrm{for every } (x, \, y)\in \Omega,
\end{equation}
i.e.,~$q_{\mathrm{s}}$ is non-negative on the first and third quadrant and non-positive elsewhere.
In particular, $q_{\mathrm{s}}$ vanishes on the coordinate axes.
Firstly, we prove the existence and uniqueness of $q_{\mathrm{s}}$, for a fixed~$\lambda>0$.

\begin{lemma} \label{lemma:existence}
 For any~$\lambda > 0$, there exists an OR or saddle solution~$q_{\mathrm{s}}$ for Problem~\eqref{AC}.
 This solution satisfies~$-B/2C \leq q_{\mathrm{s}} \leq B/2C$.
\end{lemma}
\begin{proof}
 Let~$Q$ be the truncated quadrant
 \begin{equation} \label{Q}
  Q := \bigg\{ (x, \, y)\in \Omega\colon x > 0, \ y > 0 \bigg\}.
 \end{equation}
 We 
 impose boundary conditions
 \begin{equation} \label{BC_Q}
  q = q_{\mathrm{b}} \quad \textrm{on } \partial Q \cap \partial \Omega, \qquad
  q = 0 \quad \textrm{on } \partial Q \setminus \partial{\Omega}.
 \end{equation}
 As the boundary datum is continuous and piecewise of class~$C^1$, there exist functions~$q\in W^{1,2}(Q)$ that satisfy~\eqref{BC_Q}.
 By standard arguments, we find a global minimizer, $q_{\mathrm{s}}\in W^{1,2}(Q)$, of~$H$ over~$Q$. 
 We note that $H[q_{\mathrm{s}}] = H[|q_{\mathrm{s}}|]$ and can, hence, assume that~$q_{\mathrm{s}}\geq 0$ a.e. on~$Q$.
 
 We define a function on~$\Omega$ by odd reflection of~$q_{\mathrm{s}}$ about the coordinate axes.
 The new function, still denoted by~$q_{\mathrm{s}}$, satisfies the sign condition~\eqref{sign_condition} and is a weak solution of~\eqref{AC}
 on~$\Omega\setminus\{0\}$.
 This function has bounded gradient for fixed $\lambda$ i.e.  $|\nabla q_{\mathrm{s}}|\leq C$ and we can then repeat the arguments in~\cite[Theorem~3]{fife} 
 to deduce that~$q_{\mathrm{s}}$ is a weak solution of~\eqref{AC} on~$\Omega$ (including the origin) for fixed~$\lambda$.
 By elliptic regularity on convex polygons (see e.g.~\cite[Chapter~3]{Grisvard}),
 we have that~$q_{\mathrm{s}}\in C^2(\Omega)\cap C(\overline{\Omega})$ is a classical solution of~\eqref{AC}.
 Finally, the bound~$-B/2C \leq q_{\mathrm{s}} \leq B/2C$ is a direct consequence of the maximum principle.
\end{proof}

\begin{lemma} \label{lemma:uniqueness}
 For all~$\lambda > 0$, there is at most one \emph{non-negative} solution~$q\in C^2(Q)\cap C(\overline{Q})$ to the problem
 \begin{equation} \label{AC_Q} \tag{AC\ensuremath{^\prime}}
  \begin{cases}
   -\Delta q + \dfrac{L}{\lambda^2} f(q) = 0 & \textrm{on Q} \\
   q = q_{\mathrm{b}} & \textrm{on }  Q \cap \partial \Omega \\
   q = 0 & \textrm{on } \partial Q \setminus\partial\Omega.
  \end{cases}
 \end{equation}
 There is a unique non-negative OR solution, $q_{\mathrm{s}}: \Omega \to \mathbb{R}$, defined in terms of $q$ above.
\end{lemma}
\begin{proof}
 Our proof is analogous to~\cite[Lemma~1]{fife}.
 Consider two non-negative solutions~$q_1$, $q_2$ to~\eqref{AC_Q}.
 Then~$q := \max\{q_1, \, q_2\} $ is a weak subsolution to Problem~\eqref{AC_Q} (see e.g.~\cite{GilbargTrudinger}),
 i.e.~$q\in H^1(Q)$ and 
 \begin{gather*}
  \int_Q \left(\nabla q \cdot \nabla\varphi + \dfrac{\lambda^2}{L} f(q) \varphi \right) \,\d A \leq 0 
  \qquad \textrm{for any } \varphi\in H^1_0(Q) \textrm{ s.t. } \varphi\geq 0 \\
  q \leq q_{\mathrm{b}} \qquad \textrm{on } \partial Q\cap\partial\Omega, \qquad
  q \leq 0 \qquad \textrm{on } \partial Q\setminus\partial\Omega.
 \end{gather*}
 The maximum principle, applied to both~$q_1$ and~$q_2$, implies that~$q \leq B/2C$, 
 so the constant~$B/2C$ is a supersolution to~\eqref{AC_Q}.
 Therefore, by the classical sub- and supersolution method (see, e.g.,~\cite[Theorem~1 p.~508]{evans}), 
 there exists a solution~$p_2$ of~\eqref{AC_Q} such that~$q \leq p_2 \leq B/2C$, so that
 \begin{equation} \label{uniqueness1}
  0 \leq q_1 \leq p_2 \qquad \textrm{on } Q.
 \end{equation}
 We multiply the equation for~$p_2$ with~$q_1$, multiply the equation for~$q_1$ with~$p_2$, 
 integrate by parts and take the difference to obtain
 \[
  \dfrac{\lambda^2}{L} \int_Q \bigg( f(p_2)q_1 - f(q_1)p_2 \bigg) \, \d A = 
  \int_{\partial Q} \left(\frac{\partial p_2}{\partial\nvec}q_1 - \frac{\partial q_1}{\partial\nvec}p_2 \right) \, \d \mathrm{s},
 \]
 where~$\nvec$ is the outward normal to~$\partial Q$. Recalling the definition~\eqref{f} of~$f$ and the boundary conditions~\eqref{BC_Q}, we have
 \[
  \dfrac{2C\lambda^2}{L} \int_Q q_1 p_2 \left( p_2^2 - q_1^2 \right) \,\d A
  = \int_{\partial Q} q_{\mathrm{b}} \left(\frac{\partial p_2}{\partial\nvec} - \frac{\partial q_1}{\partial\nvec} \right) \,\d\mathrm{s}.
 \]
 The left-hand side is non-negative, because of~\eqref{uniqueness1}, while the right-hand side is non-positive
 (due to~\eqref{uniqueness1} and $p_2 = q_1$ on~$\partial Q$). Therefore, both sides of the equality must vanish and
 we deduce
 \[
  q_1 = p_2 \geq \max\{q_1, \, q_2\} \qquad \textrm{on } Q
 \]
 and, in particular, $q_1 \geq q_2$ on~$Q$. By a symmetric argument, we obtain~$q_1 \leq q_2$ on~$Q$ and the conclusion follows.
 
 We can repeat the arguments of Lemma~\ref{lemma:uniqueness} on the remaining three quadrants to deduce that the OR solution is unique on $\Omega$.
\end{proof}

The choice of the Dirichlet condition on the short edges (see~\eqref{eq:d2}) in terms of the function~$g$ defined in~\eqref{eq:g}
is motivated by the fact that~$g$ satisfies the inequality 
\begin{equation}\label{g2}
  - Lg^{\prime\prime} + \frac{\lambda^2B^2}{C}g - \frac{\lambda^2B^3}{2C^2} = 0, \quad g^\prime \geq 0 
  \quad \textrm{and} \quad 0 < g < \frac{B}{2C} \qquad \textrm{on } (0, \, \varepsilon),
\end{equation}
and hence,
\begin{equation} \label{g3}
 -\lambda^2 f(g) = \frac{\lambda^2 B^2}{2C} g \left(1 + \frac{2C}{B}g\right) \left(1 - \frac{2C}{B}g\right)
 \leq \frac{\lambda^2 B^3}{2 C^2} \left(1 - \frac{2C}{B}g\right) = - L g^{\prime\prime}.
\end{equation}
We can hence, use $g$ to construct supersolutions for the Problem~\eqref{AC_Q}, as demonstrated in the following lemmas.

\begin{lemma} \label{lemma:normal_derivative}
 For any~$\lambda >0$, we have
 \[
  \frac{\partial q_{\mathrm{s}}}{\partial\nvec} \geq 0 \qquad \textrm{on } \left\{(x, \, y)\in \partial\Omega \colon x > 0, \ y > 0 \right\},
 \]
 where~$\nvec$ is the outward-pointing normal to~$\Omega$. In fact, we have the strict inequality
 \begin{equation} \label{normal_derivative_strict}
  \frac{\partial q_{\mathrm{s}}}{\partial\nvec} > 0 \qquad \textrm{on }  
  \Gamma := \left\{(x, \, y)\in \partial\Omega \colon x > \varepsilon, \ y > \varepsilon \right\}.
 \end{equation}
\end{lemma}
\begin{proof}
 We consider the function~$p\colon \overline{Q}\to\Rr$ defined by
 \begin{equation*}
  p(x, y):= \begin{cases}
             g(y) & \textrm{if } 0 \leq y < \varepsilon \\
             {B}/{2C} & \textrm{if } \varepsilon \leq y \leq 1 - \varepsilon.
            \end{cases}
 \end{equation*}
 It is straightforward to check that $p$ is a weak supersolution for Problem~\eqref{AC_Q}, from~\eqref{g2}.
 Moreover, $p\geq 0$ and the constant, $0$, is a subsolution for problem~\eqref{AC_Q}.
 The standard sub- and supersolution method, 
 combined with the uniqueness result in Lemma~\ref{lemma:uniqueness}, implies that ~$q_{\mathrm{s}} \leq p$ on~$Q$.
 Moreover, $q_{\mathrm{s}} =  p$ on
 \[
  \Gamma_x := \left\{ (x, \, y)\in\partial\Omega \colon x > \varepsilon, \ y > 0 \right\},
 \]
 and hence,
 \[
  \frac{\partial q_{\mathrm{s}}}{\partial\nvec} \geq \frac{\partial p}{\partial\nvec}  \qquad \textrm{on } \Gamma_x.
 \]
 Since~$\partial p/\partial\nvec = 0$ on~$\Gamma_x$, we conclude that~$\partial q_{\mathrm{s}}/\partial\nvec\geq 0$ on~$\Gamma_x$.
 A symmetric argument yields
 \[
  \frac{\partial q_{\mathrm{s}}}{\partial\nvec} \geq 0 \qquad \textrm{on }  
  \Gamma_y := \left\{(x, \, y)\in \partial\Omega \colon x > 0, \ y > \varepsilon \right\}.
 \]
 Finally, we notice that~$w:= B/2C - q_{\mathrm{s}}$ satisfies~$0 \leq w \leq B/2C$,
 \[
  - \Delta w + \frac{2C\lambda^2}{L} w \underbrace{\left(\frac{B}{2C} - w\right) \left(\frac{B}{C} - w\right)}
  _{\geq 0} = 0 \qquad \textrm{on } Q,
 \]
 and $w$ attains its minimum value~$w = 0$ at each point of~$\Gamma$.
 Then, the Hopf lemma (see e.g.~\cite[Lemma~3.4 p.~34]{GilbargTrudinger}) yields~$\partial w/\partial\nvec < 0$ on~$\Gamma$,
 whence~\eqref{normal_derivative_strict} follows. 
\end{proof}

Now, we use Lemma~\ref{lemma:normal_derivative} to show that~$q_{\mathrm{s}}$ can be extended to a weak supersolution~$p_{\mathrm{s}}$
to the Allen-Cahn equation~\eqref{AC}, defined on the infinite open quadrant~$K := (0, \, +\infty)^2$.
The function~$p_{\mathrm{s}}$ is constructed as follows.
We divide the quadrant~$K$ into four parts: $K_0 := Q$, $K_1$ and~$K_2$ are two semi-infinite strips
\[
 K_1 := [1 - \varepsilon, \, +\infty)\times[0, \, \varepsilon], \qquad K_2 := [0, \, \varepsilon]\times[1 - \varepsilon, \, +\infty)
\]
and~$K_3 := K \setminus (K_0 \cup K_1 \cup K_2)$ and we define~$p_{\mathrm{s}}\colon K\to\Rr$ by
\begin{equation} \label{supersolution}
 p_{\mathrm{s}}(x, \, y) := \begin{cases}
                  q_{\mathrm{s}}(x, \, y) & \textrm{on } K_0 = Q \\
                  g(y)         & \textrm{on } K_1 \\
                  g(x)         & \textrm{on } K_2 \\
                  {B}/{2C} & \textrm{on } K_3.
                 \end{cases}
\end{equation}

\begin{lemma} \label{lemma:supersolution}
 The function~$p_{\mathrm{s}}\in W^{1,2}_{\mathrm{loc}}(K)$ is a weak supersolution of the Allen-Cahn equation, that is, 
 for every non-negative function~$\varphi\in C^1_{\mathrm{c}}(K)$, we have
 \[
  \int_K \bigg\{  \nabla p_{\mathrm{s}} \cdot\nabla\varphi + \frac{\lambda^2}{L} f(p_{\mathrm{s}}) \varphi\bigg\} \, \d A \geq 0.
 \]
\end{lemma}
This lemma follows directly from Lemma~\ref{lemma:normal_derivative} and the assumptions~\eqref{g2}--\eqref{g3} on the boundary datum~$g$.
The details of the proof are omitted, for the sake of brevity.

We conclude this section by studying the first derivatives of the OR solution, $q_{\mathrm{s}}$, on the quadrant $Q$;
by symmetry, this gives us information on the derivatives of $q_{\mathrm{s}}$ on the entire domain, $\Omega$.

\begin{lemma} \label{lemma:sign_derivative}
 For any~$\lambda > 0$, the OR solution is monotonically increasing in the $x$ and $y$ directions, on the quadrant $Q$:
 \[
  \frac{\partial q_{\mathrm{s}}}{\partial x} > 0, \quad \frac{\partial q_{\mathrm{s}}}{\partial y} > 0 \qquad \textrm{on } Q.
 \] 
\end{lemma}
\begin{proof}
 This argument is inspired by~\cite[Theorem~2]{fife}.
 We first claim that~$q_{\mathrm{s}}$ in non-decreasing in the~$x$-direction, that is,
 \begin{equation} \label{sign_derivative1}
  q_{\mathrm{s}}(x, \, y) \leq q_{\mathrm{s}}(x + \tau, \, y) \qquad \textrm{for any } \tau > 0, \
  (x, \, y)\in Q \textrm{ s.t. } (x + \tau, \, y)\in Q.
 \end{equation}
 Let~$Q_\tau$ be the translated domain~$Q_\tau := Q + (\tau, \, 0)$.
 We consider Problem~(AC$^\prime_\tau$), i.e. the analogue of Problem~\eqref{AC_Q} on the translated domain~$Q_\tau$.
 By translation invariance, the unique non-negative solution to Problem~(AC$^\prime_\tau$) is given by
 \begin{equation} \label{sign_derivative2}
  q_\tau(x + \tau, \, y) :=q_{\mathrm{s}}(x, \, y) \qquad \textrm{for any } (x, \, y)\in Q.
 \end{equation}
 Moreover, from Lemma~\ref{lemma:supersolution}, the function~$p_{\mathrm{s}}$ in~\eqref{supersolution}
 is a non-negative supersolution of~(AC$^\prime_\tau$).
 The sub- and supersolution method combined with the uniqueness of solutions for
 Problem~(AC$^\prime_\tau$) (Lemma~\ref{lemma:uniqueness}) implies that
 \[
  q_\tau(x + \tau, \, y)\leq p_{\mathrm{s}}(x + \tau, \, y) \qquad \textrm{for any }  (x, \, y)\in Q  \textrm{ s.t. }  (x, \, y)\in Q .
 \]
 Recalling~\eqref{sign_derivative2} and using that~$p_{\mathrm{s}} = q_{\mathrm{s}}$ on~$Q$, we conclude the proof of~\eqref{sign_derivative1}.
 
 Next, let us set~$u := \partial q_{\mathrm{s}}/\partial x$. By~\eqref{sign_derivative1}, we know that~$u \geq 0$ on~$Q$;
 we want to prove that the strict inequality holds. We differentiate Equation~\eqref{AC_Q}:
 \[
   \Delta u - \frac{\lambda^2}{L}f^\prime(q_{\mathrm{s}}) u = 0 \qquad \textrm{on } Q.
 \]
 By the strong maximum principle we deduce that either~$u\equiv 0$ in~$Q$ (that is, $q_{\mathrm{s}}$ only depends of~$y$) 
 or~$u>0$ in~$Q$. The first possibility is clearly inconsistent with the boundary conditions, therefore~$u$ must be strictly positive inside~$Q$.  
 A similar argument can be applied to the derivative with respect to~$y$.
\end{proof}

We combine the results from Sections~\ref{sec:reduction} and \ref{sec:OR} to state the following.

\begin{lemma}
\label{lem:ORdefn}
 We define an OR LdG critical point of the energy~\eqref{eq:rescaled} on~$\Omega$, at a fixed temperature $A=-{B^2}/{3C}$, 
 subject to the Dirichlet conditions~\eqref{eq:bc1} and~\eqref{eq:bc2} to be
 \begin{equation} \label{eq:ordefn}
  \Qvec_{\mathrm{s}}(x, \, y) := q_{\mathrm{s}}(x, y)\left(\nvec_1 \otimes \nvec_1 - \nvec_2\otimes \nvec_2 \right) 
  - \frac{B}{6C}\left( 2 \hat{\mathbf{z}}\otimes\hat{\mathbf{z}} - \nvec_1 \otimes \nvec_1 - \nvec_2\otimes \nvec_2\right)
 \end{equation}
 where $q_{\mathrm{s}}$ is the OR or saddle solution of Problem~\eqref{AC_Q}, defined in Lemma~\ref{lemma:existence} or equivalently,
 $q_{\mathrm{s}}$ is a critical point of the functional $H$ defined in Proposition~\ref{prop:2}. The critical point, $\Qvec_{\mathrm{s}}$,
 exists for all $\lambda$ and all $L$ and is the unique LdG critical point, and hence globally stable, for sufficiently small $\lambda$.
\end{lemma}

\textit{Comment on proof}.
 Lemma~\ref{lem:ORdefn} is immediate from Proposition~\ref{prop:1}, \ref{prop:2} and Lem\-ma~\ref{lemma:existence}.
 The uniqueness follows from Lemmas~\ref{lemma:existence} and~\ref{lem:1}.
\endproof

\section{Instability of the OR LdG critical point}
\label{sec:ORinstability}

This section is devoted to the stability of the OR LdG critical point in Lemma~\ref{lem:ORdefn}. 
We study the stability of~$\Qvec_{\mathrm{s}}$ in terms of the stability of $q_{\mathrm{s}}$ as a critical point 
of the functional~$H$ defined in Proposition~\ref{prop:2}.
In this section, when we need to stress the dependence~$\lambda$, we write~$\qs$ and~$H_\lambda$ instead of~$q_{\mathrm{s}}$,~$H$.
The final aim of this section is to prove
\begin{theorem} \label{th:bifurcation}
 There exists a unique value~$\lambda_{\mathrm{c}} > 0$ such that a pitchfork bifurcation 
 arises at~$(\lambda_{\mathrm{c}}, \, q_{\mathrm{s}, \lambda_{\mathrm{c}}})$. More precisely, there exist positive numbers~$\epsilon$, $\delta$
 and two smooth maps
 \[
  t\in(-\delta, \, \delta) \mapsto \lambda(t)\in(\lambda_{\mathrm{c}} - \epsilon, \, \lambda_{\mathrm{c}} + \epsilon), \qquad
  t\in(-\delta, \, \delta) \mapsto h_t\in H^1_0(\Omega)
 \]
 such that all the pairs~$(\lambda, \, q)\in\Rr^+\times H^1(\Omega)$ satisfying
 \[
  q \textrm{ is a solution to \eqref{AC}}, \qquad
  |\lambda - \lambda_{\mathrm{c}}| \leq \epsilon, \qquad
  \|q - q_{\mathrm{s}, \lambda_{\mathrm{c}}}\|_{H^1(\Omega)} \leq \epsilon
 \]
 are either
 \[
  (\lambda, \, q) = (\lambda, \, q_{\mathrm{s}, \lambda}) \qquad \textrm{or} \qquad \begin{cases}
                                                                 \lambda = \lambda(t) \\
                                                                 q = q_{\mathrm{s}, \lambda(t)} + t\eta_{\lambda_{\mathrm{c}}} + t^2 h_t.
                                                                \end{cases}
 \]
 Here~$\eta_{\lambda_{\mathrm{c}}}\in H^1_0(\Omega)$ is an eigenfunction corresponding to the loss of stability at~$\lambda_{\mathrm{c}}$,
 that is, $\eta_{\lambda_{\mathrm{c}}}\not\equiv 0$ is a solution of
 \[
   \Delta \eta_{\lambda_{\mathrm{c}}} = \frac{\lambda_{\mathrm{c}}^2}{L}\left(6C q_{\mathrm{s}, \lambda_{\mathrm{c}}}^2 
  - \frac{B^2}{2C}\right)\eta_{\lambda_{\mathrm{c}}} \qquad \textrm{on } \Omega.
 \]
 Moreover, there holds
 \[
  \lambda(-t) = \lambda(t),  \qquad h_{-t}(x, \, y) = h_t(-x, \, y) \quad \textrm{for a.e. }(x, \, y)\in\Omega. 
 \]
\end{theorem}

The proof follows the same paradigm as \cite[Theorem~5.2]{lamy2014} and we address the necessary technical differences 
since the author studies a one-dimensional order reconstruction problem in~\cite{lamy2014} and we have a two-dimensional problem at hand.

We introduce some notation for functional spaces. We denote by
\begin{equation*} \label{X}
  X := \left\{q\in H^1(\Omega) \colon q = q_{\mathrm{b}} \ \textrm{on } \partial\Omega \right\}
\end{equation*}
the space of admissible functions, i.e., functions of finite energy that satisfy the Dirichlet boundary conditions (in the sense of traces).
We also define the spaces~$Y$, $Y_0$ as follows: 
\[
 Y := \big\{q\in X\colon xy \, q(x, \, y) \geq 0 \ \textrm{for a.e. } (x, \, y)\in\Omega \big\}.
\]
The space~$Y_0$ is defined in a similar way, except that the condition~$q\in X$ is replaced by~$q\in H^1_0(\Omega)$.
Every function in~$Y$ or~$Y_0$ vanishes along the axes, in the sense of traces.
The space~$Y$ is a closed affine subspace of~$X$, whose direction is~$Y_0$.
We always endow~$X$, $Y$, $Y_0$ with the topology induced by the~$H^1$-norm.

Lemma~\ref{lemma:uniqueness} implies that~$\qs$ is the only solution of~\eqref{AC} that belongs to~$Y$.
The stability of~$\qs$ is measured by the quantity
\begin{equation} \label{mu}
 \mu(\lambda) := \inf_{\eta \in H^1_0(\Omega)\setminus\{0\}}\dfrac{\delta^2 H_\lambda[\eta]}{\int_\Omega\eta^2},
\end{equation}
where~$\delta^2 H_\lambda$ is the second variation of~$H_\lambda$ at~$\qs$, given by
\begin{equation} \label{second_variation}
 \delta^2 H_\lambda[\eta] := \frac{\d^2 }{\d t^2}_{| t = 0} H_\lambda[\qs + t\eta] 
 = \int_\Omega \left\{ \abs{\nabla\eta}^2  + \frac{\lambda^2}{L} \left(6C\qs^2 - \frac{B^2}{2C}\right)\eta^2 \right\} \,\d A .
\end{equation}
In other words, $\mu(\lambda)$ is the first eigenvalue of~$\delta^2 H_\lambda$.
By a standard application of the maximum principle, 
one sees that if~$\eta\not\equiv 0$ is an eigenfunction associated with~$\mu(\lambda)$ 
(that is, a minimizer for Problem~\eqref{mu}) then~$\eta$ has the same sign everywhere on~$\Omega$.
Consequently, $\mu(\lambda)$ is a simple eigenvalue for any~$\lambda > 0$.

\begin{lemma} \label{lemma:smoothness}
 The map~$(0, \, +\infty)\to Y$ defined by~$\lambda \mapsto \qs$ is smooth.
\end{lemma}
\begin{proof}
 We claim that, for any~$\lambda > 0$, there exists a positive constant~$\alpha(\lambda)$ such that
 \begin{equation} \label{smoothness1}
  \delta^2 H_\lambda[\eta] \geq \alpha(\lambda) \int_\Omega \left|\nabla\eta \right|^2 \, \d A
  \qquad \textrm{for any } \eta\in Y_0.
 \end{equation}
 Once we prove the inequality~\eqref{smoothness1}, we can apply the implicit function theorem as in~\cite[proof of Proposition~4.3]{lamy2014}
 to obtain that~$\lambda\mapsto \qs$ is smooth.
 We prove~\eqref{smoothness1} with the help of an Hardy-type trick.  Fix $\lambda > 0$ and a test function~$\eta\in Y_0$.
 (Now~$\lambda$ is fixed, so we omit in the notation for~$q_{\mathrm{s}}$ and~$H$.)
 The function~$\eta$ vanishes along the square diagonals or the coordinate axes.
 By an approximation argument, we can assume WLOG that~$\eta$ is smooth and its support does not intersect the square diagonals.
 In particular, we have~$q_{\mathrm{s}}(x, \, y)\neq 0$ for any~$(x, \, y)\in\mathrm{support}(\eta)$. 
 Therefore, there exists a smooth function~$v\colon\overline\Omega\to\Rr$ such that~$\eta = q_{\mathrm{s}} v$.
 Substituting $\eta = q_{\mathrm{s}} v$ into~\eqref{second_variation} and using~\eqref{AC}, we obtain
 \begin{equation*}
  \begin{split}
   \delta^2 H[\eta] = \int_\Omega \bigg\{ \left( \left|\nabla q_{\mathrm{s}} \right|^2v^2 + 2 q_{\mathrm{s}} v \nabla q_{\mathrm{s}} \cdot \nabla v
    + \left| \nabla v \right|^2 q_{\mathrm{s}}^2\right) 
    + \left(\Delta q_{\mathrm{s}} + 4 \frac{\lambda^2}{L} Cq_{\mathrm{s}}^3\right)q_{\mathrm{s}} v^2 \bigg\} \, \d A  .
  \end{split}
 \end{equation*}
 An integration by parts yields
 \[
  \int_\Omega q_{\mathrm{s}} v^2 \Delta q_{\mathrm{s}} \, \d A  = 
  - \int_\Omega \left\{ \left|\nabla q_{\mathrm{s}} \right|^2 v^2 + 2 q_{\mathrm{s}} v \nabla q_{\mathrm{s}}\cdot \nabla v\right\} \, \d A  .
 \]
 All the boundary terms vanish because $v = \eta = 0$ on~$\partial\Omega$. Therefore, we have
 \[
  \delta^2 H[\eta] = \int_\Omega \left\{\left|\nabla v\right|^2 + \frac{4\lambda^2C}{L} q_{\mathrm{s}}^2 v^2 \right\} q_{\mathrm{s}}^2 \, \d A 
 \]
and therefore, $\delta^2 H_\lambda[\eta] \geq 0$ for any~$\eta\in Y_0$, with equality if and only if~$\eta = 0$.
The inequality~\eqref{smoothness1} now follows from a standard argument 
(e.g., one can argue by contradiction and use the compact embedding of~$Y_0$ into~$L^2(\Omega)$).
\end{proof}

We now consider the re-scaled functions~$\tildeqs\colon\lambda\Omega\to\Rr$ defined by
\[
 \tildeqs(x, \, y) := \qs\left(\frac{x}{\lambda}, \, \frac{y}{\lambda}\right) \qquad \textrm{for } (x, \, y)\in\lambda\Omega,
\]
which satisfy the equation 
\begin{equation} \label{scaled_AC}
 - \Delta\tildeqs + \frac{1}{L}f(\tildeqs) = 0 \qquad \textrm{on } \lambda\Omega,
\end{equation}
with appropriate boundary conditions for $\tildeqs$. 
The truncated quadrant~$\lambda Q$, is the intersection of~$\lambda\Omega$ 
with the first coordinate quadrant.

\begin{lemma} \label{lemma:monotonicity_q}
 For any~$\lambda > 0$ and any~$(x, \, y)\in\lambda Q$, we have
 \[
  \frac{\partial\tildeqs}{\partial\lambda}(x, \, y) < 0.
 \]
\end{lemma}
\begin{proof}
 We first show that~$\tildeqs$ is non-increasing as a function of~$\lambda$.
 We take~$\lambda_1 < \lambda_2$ and, for simplicity, we set~$\tilde{q}_j := \tilde{q}_{\mathrm{s}, \lambda_j}$ for~$j\in\{1,\, 2\}$.
 We claim that
 \begin{equation} \label{monotoneq1}
  \tilde{q}_1 (x, \, y) \geq \tilde{q}_2 (x, \, y) \qquad \textrm{for any } (x,\, y)\in \lambda_1\Omega.
 \end{equation}
 We extend~$\tilde{q}_1$ to a new function~$\tilde{p}_1$, defined over the infinite quadrant~$K = (0, \, +\infty)^2$,
 so that~$\tilde{p}_1$ is a weak supersolution of~\eqref{scaled_AC} on~$K$.
 We define~$\tilde{p}_1$ as in Equation~\eqref{supersolution}, with the obvious modifications due to the scaling $\Omega\to\lambda\Omega$.
 Then the function~$\tilde{p}_1$  satisfies
 \[
  \begin{cases}
   \displaystyle\int_{\lambda_2\Omega} \left(\nabla\tilde{p}_1\cdot\nabla\varphi + \frac{1}{L} f(\tilde{p}_1)\varphi \right) \, \d A \geq 0
   & \textrm{for any } \varphi\in H^1_0(\lambda_2 Q), \ \varphi\geq 0 \\
   \tilde{p}_1 \geq \tilde{q}_2 & \textrm{on } \partial (\lambda_2 Q) .
  \end{cases}
 \]
 The sub- and supersolution method, combined with the uniqueness result in Lem\-ma~\ref{lemma:uniqueness} and a scaling argument, 
 imply that $\tilde{p}_1 \geq \tilde{q}_2$ on~$\lambda_2\Omega$.
 Since $\tilde{p}_1 = \tilde{q}_1$ on~$\lambda_1\Omega$, we conclude that~\eqref{monotoneq1} holds.
 
 Now, fix~$\lambda > 0$. It follows form~\eqref{monotoneq1} that~$v := \partial\tildeqs/\partial\lambda\leq 0$ on~$\lambda\Omega$.
 We differentiate Equation~\eqref{scaled_AC} with respect to~$\lambda$ and show that $v$ satisfies
 \[
  - \Delta v + \frac{1}{L} f^\prime(\tildeqs) v = 0 \qquad \textrm{on } \lambda\Omega.
 \]
 By the strong maximum principle, 
 we conclude that either~$v\equiv 0$ on~$\lambda\Omega$ or~$v < 0$ on~$\lambda\Omega$.
 However, $\tildeqs$ satisfies the boundary condition
 \[
  \tildeqs \left(\frac{\lambda}{2}, \, \frac{\lambda}{2}\right) = \qs \left(\frac{1}{2}, \, \frac{1}{2}\right) = \frac{B}{2C}.
 \]
 Differentiating with respect to~$\lambda$, we obtain
 \[
  v \left(\frac{\lambda}{2}, \, \frac{\lambda}{2}\right) 
  + \frac{1}{2\sqrt{2}} \frac{\partial\tildeqs}{\partial\nvec}\left(\frac{\lambda}{2}, \, \frac{\lambda}{2}\right) = 0.
 \]
 Since the normal derivative is strictly positive
 (by Lemma~\ref{lemma:normal_derivative} and a scaling argument), 
 we conclude that~$v(\lambda/2, \, \lambda/2) < 0$.
 Therefore, $v$ cannot vanish identically and, by the strong maximum principle, 
 it must be strictly negative everywhere on~$\lambda\Omega$.
 \end{proof}

\begin{lemma} \label{lemma:monotonicity_mu}
 The map~$(0, \, +\infty)\to\Rr$ defined by $\lambda\mapsto\mu(\lambda)$ is smooth and~$\mu^\prime(\lambda) < 0$ for any~$\lambda > 0$.
\end{lemma}
\begin{proof}
The smoothness of~$\mu$ can be proven as in~\cite[Proposition~4.3]{lamy2014}.
We now prove that~$\mu^\prime(\lambda) < 0$ for any~$\lambda>0$.
By Equation~\eqref{mu} and a scaling argument, for any~$\lambda > 0$, we have
\begin{equation} \label{mu_rescaled}
 \begin{split}
  \mu(\lambda) = \inf_{\eta}
  \int_{\lambda\Omega} \left\{ \left|\nabla\eta\right|^2  + \frac{1}{L}\left(6C\tildeqs^2 - \frac{B^2}{2C}\right)\eta^2 \right\} \, \d A .
 \end{split}
\end{equation}
The infimum is taken over all~$\eta\in H^1_0(\lambda\Omega)$ such that~$\int_{\lambda\Omega} \eta^2 = 1$.
Fix ~$\lambda_0 > 0$ and let~$\eta_0\in H^1_0(\lambda_0\Omega)$ be minimiser for~\eqref{mu_rescaled}.
We extend~$\eta_0$ to be zero outside~$\lambda_0\Omega$. Then, for any~$\lambda \geq \lambda_0$, we have
\[
 \mu(\lambda) \leq \mu_0(\lambda) :=
 \int_{\lambda_0\Omega} \left\{\left|\nabla\eta_0\right|^2  + \frac{1}{L} \left(6C\tildeqs^2 - \frac{B^2}{2C}\right)\eta_0^2 \right\} \, \d A,
\]
with equality if~$\lambda = \lambda_0$. This implies
\[
 \mu^\prime(\lambda_0) \leq \mu_0^\prime(\lambda_0) = \frac{12 C}{L} \int_{\lambda_0\Omega} 
 \tilde{q}_{\lambda_0, \mathrm{s}}\frac{\d \tildeqs}{\d \lambda}_{|\lambda = \lambda_0} \eta_0^2 \, \d A.
\]
By Lemma~\ref{lemma:monotonicity_q} and the odd symmetry of~$\qs$ about the axes, the integrand is non-positive
and it does not vanish identically. Therefore, the right-hand side is strictly negative and so is~$\mu^\prime(\lambda_0)$.
\end{proof}

\begin{lemma} \label{lemma:instability}
 There exists a positive number~$\lambda^*$ such that~$\mu(\lambda) < 0$ for any~$\lambda \geq \lambda^*$.
\end{lemma}
\begin{proof}
 The uniform bound~$|q_{\mathrm{s}, \lambda}|\leq B/2C$ and~\cite[Lemma~A.2]{BBH-degreezero}, applied to the Equation~\eqref{AC},
 yield the estimate~$|\nabla q_{\mathrm{s}, \, \lambda}|\leq C \lambda$ for some $\lambda$-independent constant~$C$.
 By scaling, we obtain~$|\tilde{q}_{\mathrm{s}, \lambda}| + |\nabla\tilde{q}_{\mathrm{s}, \lambda}| \leq C$.
 Therefore, by the Ascoli-Arzel\`a theorem, we have the locally uniform convergence $\tilde{q}_{\mathrm{s}, \lambda} \to \tilde{q}_\infty$
 as~$\lambda\to+\infty$, up to a non-relabelled subsequence. Taking the limit in both sides of~\ref{scaled_AC}, we see that~$\tilde{q}_\infty$ 
 is the unique saddle solution of the Allen-Cahn equation~\eqref{scaled_AC} on~$\Rr^2$ (see~\cite{fife}).
 Thanks to~\cite[Lemma~3.4]{schatzman}, we find a function~$\eta\in H^2(\Rr^2)$ such that
 \begin{equation*} 
 \int_{\Rr^2} \left\{ \abs{\nabla\eta}^2  + \frac{1}{L} \left(6C\tilde{q}_\infty^2 - \frac{B^2}{2C}\right)\eta^2 \right\} \,\d A < 0.
 \end{equation*}
 By truncation, we can assume WLOG that~$\eta$ has compact support.
 Then, using the locally uniform convergence~$\tilde{q}_{\mathrm{s}, \lambda}\to\tilde{q}_\infty$,
 we conclude that the right-hand side of~\eqref{mu_rescaled} becomes negative for~$\lambda$ large enough,
 whence~$\mu(\lambda) < 0$ for $\lambda$ large enough. 
\end{proof}

\textit{Proof of Theorem~\ref{th:bifurcation}}
 We know that~$\mu(\lambda) > 0$ for~$0 < \lambda \ll 1$:
 this follows from the stability result for the full LdG system in Lemma~\ref{lem:1}.
 (This can also be proved directly from~\eqref{mu_rescaled}, by applying Poincar\'e inequality.)
 Combining this with Lemma~\ref{lemma:instability}, we find~$\lambdac > 0$ such that $\mu(\lambdac) = 0$.
 Such a $\lambdac$ is unique, because~$\mu$ is strictly decreasing (Lemma~\ref{lemma:monotonicity_mu}).
 
 We revert to the original re-scaled domain, $\Omega$, for the rest of the computation. To show that a pitchfork bifurcation arises at~$\lambda = \lambdac$, we apply the
 Crandall and Rabinowitz bifurcation theorem~\cite[Theorem~1.7]{CrandallRabinowitz} to the map
 $\mathscr{F}\colon\Rr^+\times H^1_0(\Omega)\to H^{-1}(\Omega)$ defined by
 \[
  \mathscr{F}(\lambda, \, h) := - \Delta(\qs + h) + \frac{\lambda^2}{L} f(\qs + h)
 \]
 for any~$(\lambda, \, h)\in\Rr^+\times H^1_0(\Omega)$.
 We first have to check that the assumptions of the theorem are satisfied.
 Clearly, we have~$\mathscr{F}(\lambda, \, 0) = 0$ for any~$\lambda > 0$. The map~$\mathscr{F}$ is smooth and we have
 \[
  \D_h \mathscr{F} (\lambda, \, 0) = - \Delta + \frac{\lambda^2}{L} f^\prime(\qs).
 \]
 This is a Fredholm operator of index~$0$, whose smallest eigenvalue~$\mu(\lambda)$ has multiplicity~$1$.
 Therefore, for~$\lambda = \lambdac$, we have
 \[
  \dim\frac{H^{-1}(\Omega)}{\mathrm{range} \, \D_h \mathscr{F} (\lambdac, \, 0)} 
  = \dim \mathrm{kernel} \, \D_h \mathscr{F} (\lambdac, \, 0) = 1.
 \]
 Let~$\eta_\lambda$ be an eigenfunction associated with~$\mu(\lambda)$ (i.e., a minimiser for~\eqref{mu}), 
 renormalised so that~$\int_\Omega \eta_\lambda^2 = 1$.
 From~\cite[Lemma~1.3]{CrandallRabinowitz-ARMA}, we can assume that the
 map~$\lambda\mapsto\eta_\lambda$ is smooth, at least for~$\lambda$ close enough to~$\lambdac$.
 In order to apply Crandall and Rabinowitz's theorem, we need to check that
 \begin{equation} \label{bifurcation2}
  \D_\lambda\D_h \mathscr{F} (\lambdac, \, 0) [\eta_{\lambdac}] \notin \mathrm{range} \,  \D_h \mathscr{F}(\lambdac, 0).
 \end{equation}
 Proceeding by contradiction, assume that there exists~$h\in H^1_0(\Omega)$ such that
 \[
  \D_\lambda\D_h \mathscr{F} (\lambdac, \, 0) [\eta_{\lambdac}] = \D_h \mathscr{F}(\lambdac, 0)[h].
 \]
 Then, using the fact that~$\D_h \mathscr{F} (\lambda, \, 0)$ is symmetric 
 and~$\eta_{\lambdac}\in\ker\D_h \mathscr{F} (\lambdac, \, 0)$,
 we obtain that
 \begin{equation*}
  \begin{split}
   \mu^\prime(\lambdac) &= \frac{\d}{\d\lambda}_{|\lambda = \lambdac}
   \langle \D_h \mathscr{F} (\lambda, \, 0) [\eta_\lambda], \, \eta_\lambda \rangle \\
   &= \langle \D_\lambda\D_h \mathscr{F} (\lambdac, \, 0) [\eta_{\lambdac}], \, \eta_{\lambdac} \rangle 
   + 2\langle \D_h \mathscr{F} (\lambdac, \, 0) [\eta_{\lambdac}], \, \partial_{\lambda} \eta_{\lambda|\lambda = \lambdac} \rangle \\
   &= \langle \D_h \mathscr{F} (\lambdac, \, 0) [h], \, \eta_{\lambdac} \rangle \\
   &= \langle \D_h \mathscr{F} (\lambdac, \, 0) [\eta_{\lambdac}], \, h \rangle = 0.
  \end{split}
 \end{equation*}
 However, we know  that~$\mu^\prime(\lambdac) < 0$ by Lemma~\ref{lemma:monotonicity_mu}.
 Therefore, we have a contradiction and~\eqref{bifurcation2} cannot hold. 
 Thus, all the assumptions of Crandall and Rabinowitz's theorem are satisfied.
 
 By Crandall and Rabinowitz's theorem, we find positive numbers~$\epsilon$ and~$\delta$
 such that any pair~$(\lambda, \, q)\in\Rr^+\times H^1(\Omega)$ satisfying
 \[
  q \textrm{ is a solution of } \eqref{AC}, \qquad
  |\lambda - \lambdac| \leq \epsilon, \, \qquad 
  \|q - q_{\mathrm{s}, \lambdac}\|_{H^1(\Omega)} \leq \epsilon
 \]
 is either of the form~$(\lambda, \, \qs)$ or 
 \begin{equation} \label{bifurcation1}
  \lambda = \lambda(t), \qquad q = q_{\mathrm{s}, \lambda(t)} + t\eta_{\lambdac} + t^2 h_t,
 \end{equation}
 where~$\lambda(t)\in (\lambdac - \epsilon, \, \lambdac + \epsilon)$ and~$h_t\in H^1(\Omega)$
 are smooth functions of the scalar parameter~$t\in (-\delta, \, \delta)$.
 (The smoothness of $\lambda(t)$, $h_t$ is given by~\cite[Theorem~1.18]{CrandallRabinowitz}.)
 
 To conclude the proof of the theorem, we remark that if~$q$ is a solution of~\eqref{AC}, then so is
 \[
  \bar{q}(x, \, y) := -q(-x, \, y).
 \]
 By construction, the order reconstruction solution satisfies~$\bar{q}_{\mathrm{s}, \lambda(t)} = q_{\mathrm{s}, \lambda(t)}$.
 Moreover, we have $\bar{\eta}_{\lambdac} = -\eta_{\lambdac}$. Indeed, both~$\bar{\eta}_{\lambdac}$
 and~$\eta_{\lambdac}$ are eigenfunctions associated with the same eigenvalue~$\mu(\lambdac)$, 
 which has multiplicity one, therefore~$\bar{\eta}_{\lambdac} = \alpha\eta_{\lambdac}$ for some real number~$\alpha$.
 By taking into account that $\int_\Omega \bar{\eta}^2_{\lambdac} \, \d A = \int_\Omega {\eta}^2_{\lambdac} \, \d A = 1$
 and that~$\eta_{\lambdac}$ vanishes nowhere on~$\Omega$ (by the maximum principle), we deduce that ~$\alpha = -1$.
 Therefore, for a solution $q = q_{\mathrm{s}, \lambda(t)} + t\eta_{\lambdac} + t^2 h_t$ of~\eqref{AC} with~$\lambda = \lambda(t)$,
 there is another associated solution of the same equation, given by
 \[
  \bar{q} = q_{\mathrm{s}, \lambda(t)} - t\eta_{\lambdac} + t^2 \bar{h}_t.
 \]
 This solution is close to the bifurcation point~$(\lambdac, \, q_{\mathrm{s}, \lambdac})$ 
 and does not belong to the branch of OR solutions; therefore, it must be of the form given by~\eqref{bifurcation1}.
 This yields
 \[
  \lambda(-t) = \lambda(t), \qquad h_{-t}(x, \, y) = \bar{h}_t(x, \, y) = -h_t(-x, \, y)
 \]
 and concludes the proof.
\endproof


\begin{corollary} \label{lem:ORLdGinstability}
 The OR LdG critical point~$\Qvec_{\mathrm{s}}$ defined in Lemma~\ref{lem:ORdefn} is unstable for $\lambda > \lambda_{\mathrm{c}}$.
\end{corollary}

\begin{proof}
 Consider a perturbation of the form
 \begin{equation} \label{eq:eta}
  \Vvec := \eta \left(\nvec_1 \otimes \nvec_1 - \nvec_2 \otimes \nvec_2 \right).
 \end{equation}
 A standard computation shows that the second variation of the LdG energy about the critical point~$\Qvec_{\mathrm{s}}$ reduces to
 \begin{equation}\label{eq:eta2}
  \delta^2 I[\Vvec] 
  = \int_{\Omega} \left\{ \left|\nabla\eta\right|^2  + \frac{\lambda^2}{L}\left(6C \qs^2 - \frac{B^2}{2C}\right)\eta^2 \right\} \,\d A .
 \end{equation}
 From Theorem~\ref{th:bifurcation} and Lemma~\ref{lemma:monotonicity_mu}, we have that for any $\lambda > \lambda_{\mathrm{c}}$ 
 there exists an admissible $\eta$ (vanishing on $\partial \Omega$) such that $\delta^2 I[\Vvec]<0$ in \eqref{eq:eta2}.
 The conclusion now follows.
\end{proof}

\section{Numerics}
\label{sec:numerics}

In this section, we perform some numerical experiments to study order reconstruction solutions on two
specific two-dimensional regular polygons --- the square and a hexagon, the latter perhaps serving to partially 
illustrate the generic nature of such solutions.

We work with the gradient flow model for nematodynamics in the LdG framework, as this is arguably the simplest model
to study the evolution of solutions without any external effects or fluid flow.
Informally speaking, gradient flow models are dictated by the principle that dynamic solutions evolve along 
a path of decreasing energy, converging to a stable equilibrium for long times \cite{gradientflow}. 
We adopt the standard gradient flow model associated with the LdG energy, described by a system of five coupled 
nonlinear parabolic partial differential equations as shown below:
\begin{equation}
\label{eq:e1}
\gamma \Qvec_t = L \Delta \Qvec - A \Qvec + B \left( \Qvec \Qvec - \frac{|\Qvec|^2}{3}\mathbf{I}\right) - C |\Qvec|^2 \Qvec
\end{equation}
where $\gamma$ is a positive rotational viscosity, $\mathbf{I}$ is the $3 \times 3$ identity matrix
and $A = -\frac{B^2}{3C}$ so that we can make comparisons between the numerics and the analysis above.

We adopt the same scalings as in \cite{front2016} i.e. non-dimensionalize the system by setting
$\bar{t} := \frac{20 t L}{\gamma \lambda^2}, \bar{\rvec} := \frac{\rvec}{\lambda}$ where $\lambda$ 
is a characteristic geometrical length scale to get
\begin{equation}
\label{eq:e2}
 \frac{\partial \Qvec}{\partial \bar{t}} = \bar{\Delta}\Qvec - \frac{\lambda^2}{L}\left(A \Qvec 
 - B \left( \Qvec \Qvec - \frac{|\Qvec|^2}{3}\mathbf{I}\right) + C |\Qvec|^2 \Qvec \right)
\end{equation}
and we drop the bars from all subsequent discussion for brevity.

\subsection{Numerics on a square}
\label{sec:square}

We first take a square centered at the origin with edge length $2 \lambda$ and impose a boundary condition of the form
\begin{equation}
\label{eq:e4}
 \Qvec_{\mathrm{b}, ij} = q \left( \xhat_i \xhat_j - \yhat_i \yhat_j \right) 
 - \frac{B}{6C}\left(2 \zhat_i \zhat_j - \xhat_i \xhat_j - \yhat_i \yhat_j \right)
\end{equation}
where $\xhat, \yhat, \zhat$ are unit-vectors in the $x$, $y$ and $z$-directions respectively and
\begin{equation} \label{eq:e3}
 \begin{aligned}
   & q(x, \, -1) = q (x, \, 1) = \frac{B}{2C} && \textrm{for } -1+\varepsilon \leq x \leq 1-\varepsilon \\ 
   & q(x, \, -1) = q(x, \, 1) = f(x)          && \textrm{otherwise}  \\
   & q(-1, \, y) = q(1, \, y) = -\frac{B}{2C} && \textrm{for } -1+\varepsilon \leq y \leq 1-\varepsilon \\
   & q(-1, \, y) = q(1, \, y) = -f(y)         && \textrm{otherwise,}
 \end{aligned}
\end{equation}
where
\[
 f(s) := \frac{B}{2C} \frac{ 1- |s|}{\varepsilon} \quad \textrm{for } |s|\leq \varepsilon.
\]
Consequently, $q$ is fixed to be zero at the vertices.
We work with a fixed initial condition of the form \eqref{eq:e4} with
\begin{equation} \label{eq:e6}
 \Qvec_0 := q_0 \left( \xhat_i \xhat_j - \yhat_i \yhat_j \right) - \frac{B}{6C}\left(2 \zhat_i \zhat_j - \xhat_i \xhat_j - \yhat_i \yhat_j \right)
\end{equation}
and
\begin{equation} \label{eq:e7}
 q_0(x, \, y) := \begin{cases}
  \dfrac{B}{2C}  & \textrm{for } -y < x < y \\
  -\dfrac{B}{2C} & \textrm{for } -x < y < x,
 \end{cases}
\end{equation}
such that $q_0 = 0$ on the diagonals $x=\pm y$.
In other words, $q_0$ mimics the saddle-type order reconstruction solution studied above
in the sense that the initial condition, $\Qvec_0$, has a constant eigenframe with an uniaxial cross,
that has negative order parameter and connects the four square vertices.

For a boundary condition and an initial condition of the form \eqref{eq:e4}--\eqref{eq:e7}, 
there is a dynamic solution, $\Qvec(\rvec, \, t)$, of the system \eqref{eq:e2} given by
\begin{equation}
\label{eq:e8}
 \Qvec(\rvec, \, t): = q(x, \, y, \, t) \left( \xhat_i \xhat_j - \yhat_i \yhat_j \right) 
 - \frac{B}{6C}\left(2 \zhat_i \zhat_j - \xhat_i \xhat_j - \yhat_i \yhat_j \right)
\end{equation} 
where the evolution of $q$ is governed by
\begin{equation}
\label{eq:e9}
\frac{\partial q}{\partial t} = \Delta q - \frac{2C \lambda^2}{L}q \left( q -  \frac{B}{2C} \right) \left( q + \frac{B}{2C} \right).
\end{equation}
In what follows, we solve the evolution equation \eqref{eq:e9} on a re-scaled square (with vertices 
at $(-1,\, -1), (-1, \, 1), (1, \, -1), (1,\, 1)$ respectively) for different values of $\lambda$. 
We use a standard finite-difference method for the spatial derivatives and the Runge-Kutta scheme for time-stepping 
in the numerical simulations on a square (also see \cite{front2016} for more discussion on numerical methods).
We expect to see that $q=0$ along $x=\pm y$ for small values of $\lambda$, since the order reconstruction solution 
is the unique LdG critical point for small $\lambda$ and we expect to see transition layers near a pair of opposite edges 
for large $\lambda$, as suggested by
Proposition~\ref{prop:3}.

Let $\bar{\lambda}^2 :=\frac{2C \lambda^2}{L}$. We let $B = 0.64 \times 10^4 \mathrm{Nm}^{-2}$, $C=0.35 \times 10^4 \mathrm{Nm}^{-2}$
throughout this section \cite{newtonmottram}. In Figures~\ref{fig:1}, \ref{fig:2}, we solve the evolution equation \eqref{eq:e9},
subject to the Dirichlet condition \eqref{eq:e3} and the initial condition, $q(x, \, y, \, 0) = q_0(x, \, y)$ 
where $q_0$ is defined in \eqref{eq:e7}, for $\bar{\lambda}^2 = 0.05$ and $\bar{\lambda}^2 = 200$ respectively. 
For $\bar{\lambda}^2 = 0.05$, the scalar profile relaxes the sharp transition layers at $x=\pm y$ but retains
the vanishing diagonal cross with $q(x,\pm x, t) = 0$ for all times. The corresponding dynamic solution, $\Qvec(\rvec, t)$ in \eqref{eq:e8}, 
has an uniaxial cross with negative order parameter, connecting the four square vertices, consistent with the stability and uniqueness 
results for the saddle-type order reconstruction solution studied in Sections~\ref{sec:OR} and \ref{sec:ORinstability}. For large values of 
$\bar{\lambda}^2 =200$, the initial condition has the diagonal cross but the diagonal cross rapidly relaxes into a pair of transition 
layers, one layer being localized near $x=-1$ (or $y=-1$) and the other layer being localized near $x=+1$ (or $y=+1$); see 
Figure~\ref{fig:2}. In Figure~\ref{fig:3}, we plot $q(0, \, 0)$ - the value of the converged solution at the origin as a function of 
$\bar{\lambda}^2$. It is clear that we have lost the diagonal cross if $q(0, \, 0) \neq 0$ and hence, 
one might reasonably deduce that the order reconstruction solution loses stability for values of $\bar{\lambda}^2$ for which
$q(0, \, 0) \neq 0$. In Figure~\ref{fig:3}, we see that $q(0, \, 0)=0$ for $\bar{\lambda}^2 \leq 9.2$ and $q(0, \, 0) \neq 0$ 
for $\bar{\lambda}^2 > 9.2$. The picture is consistent with a supercritical pitchfork bifurcation as established in 
Theorem~\ref{th:bifurcation}. In other words, we expect the order reconstruction solution to lose stability on a square domain with edge 
length $2 \lambda$ and for which
\begin{equation}
\label{eq:e10}
 \lambda^2 > 5\frac{L}{C}.
\end{equation}

\begin{figure}[h!]
	\centering
 	\centerline{\includegraphics[width=5cm]{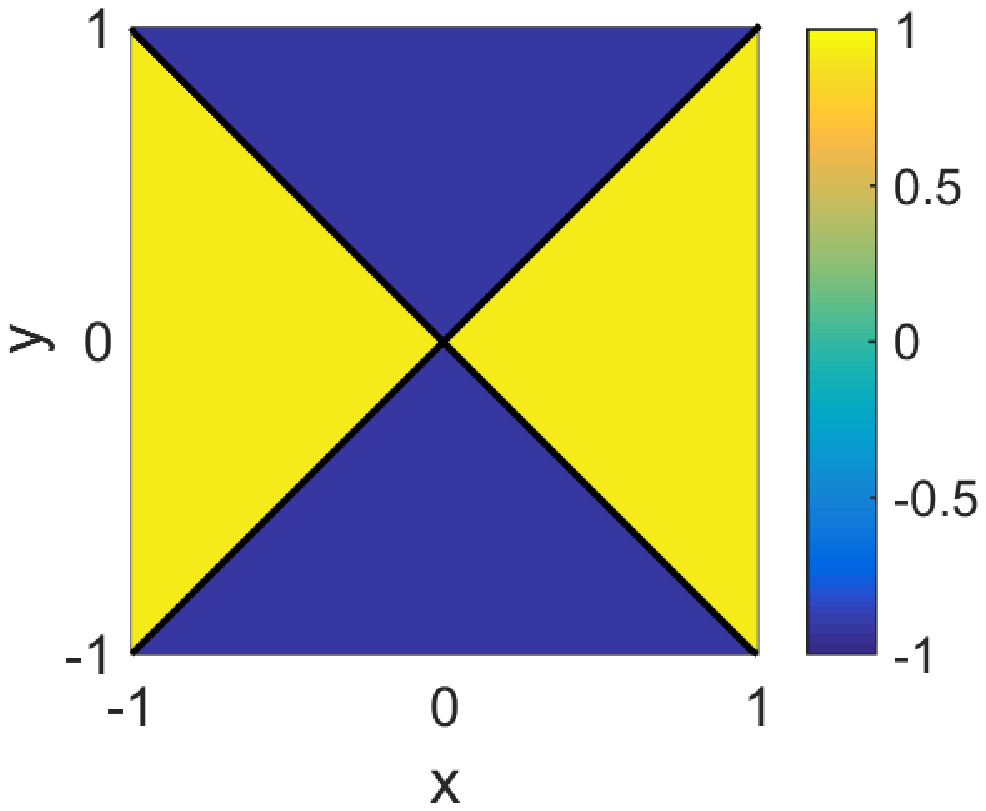}\includegraphics[width=5cm]{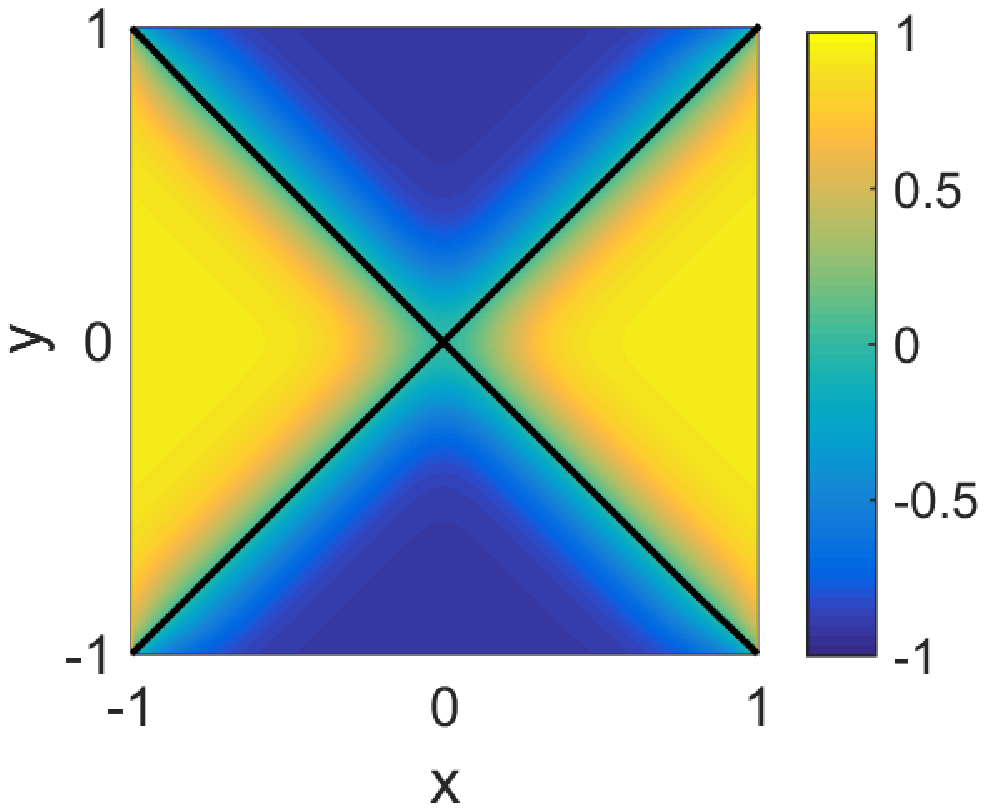}\includegraphics[width=5cm]{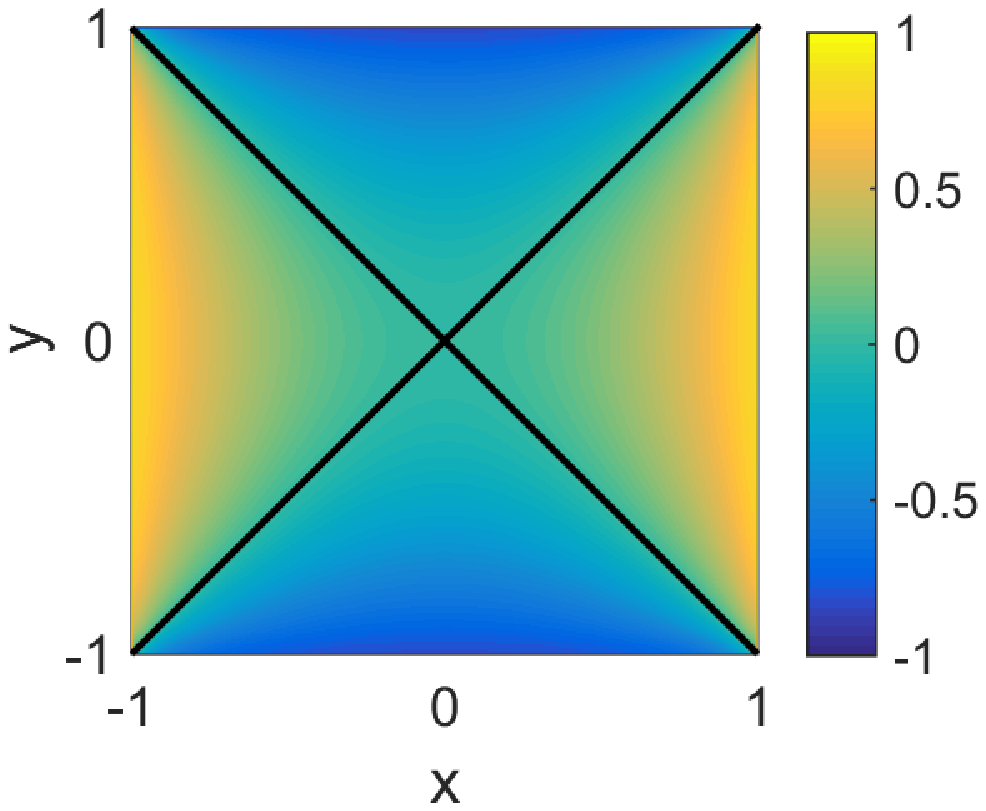}}
	\caption{$q(x,\,y,\,t)$ for $t=0$, $t=0.01$ and $t=2$ for $\lambda^2=0.05$.}
	\label{fig:1}	
\end{figure}

\begin{figure}[h!]
	\centering
 	\centerline{\includegraphics[width=5cm]{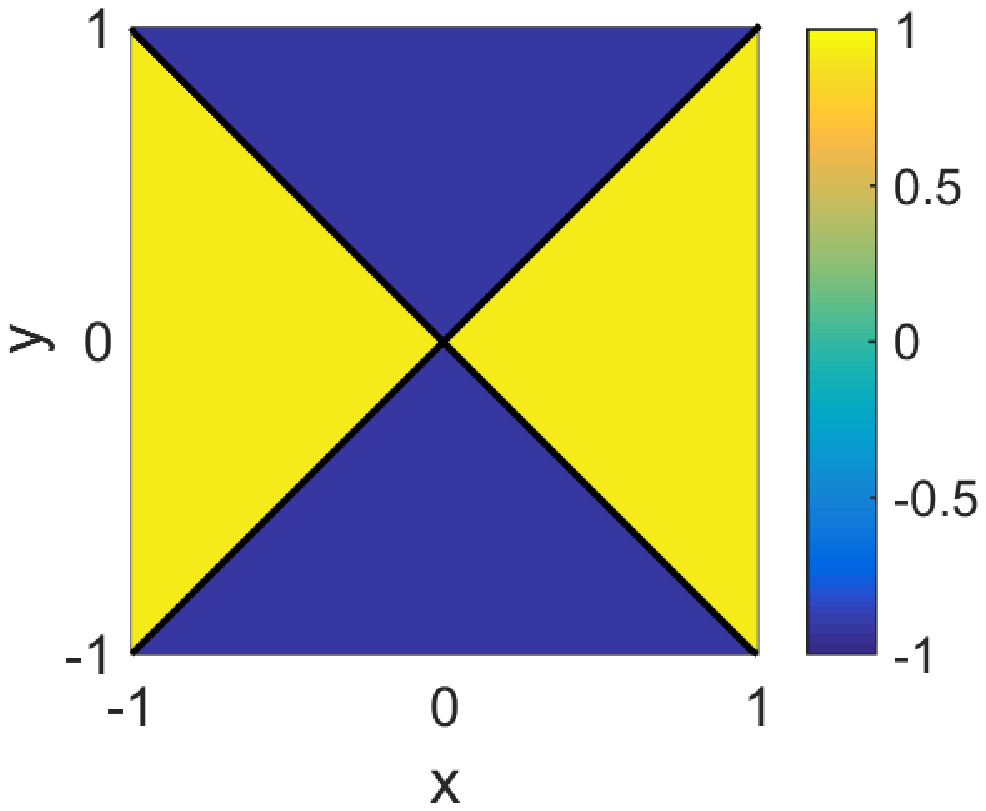}\includegraphics[width=5cm]{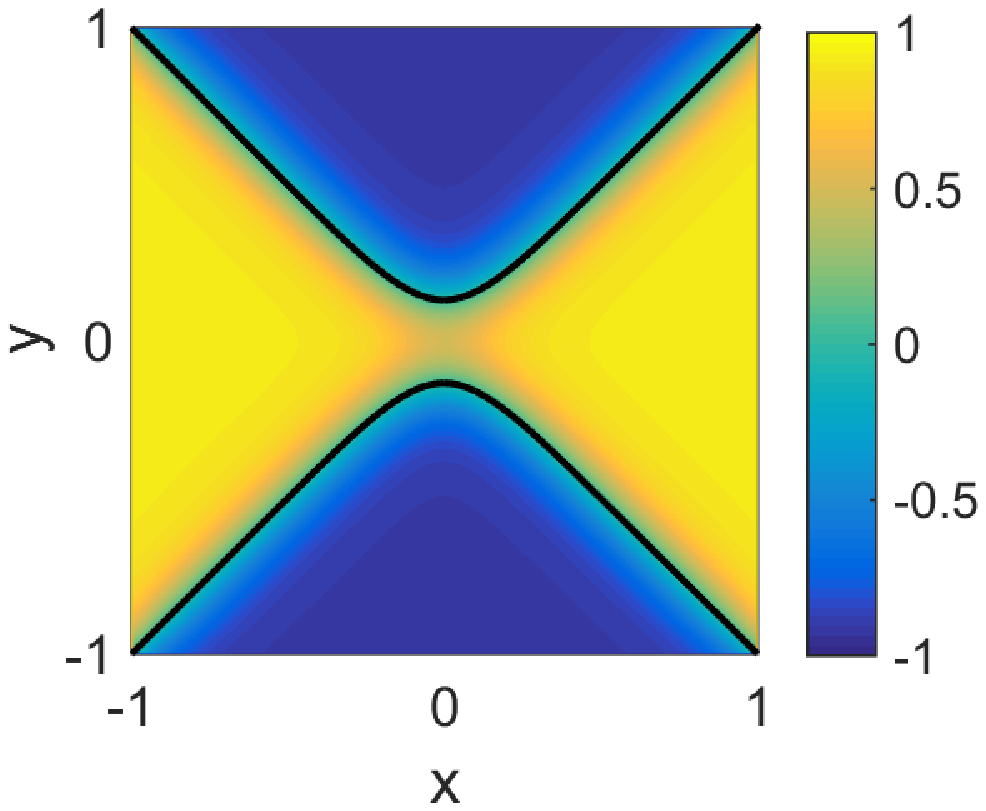}\includegraphics[width=5cm]{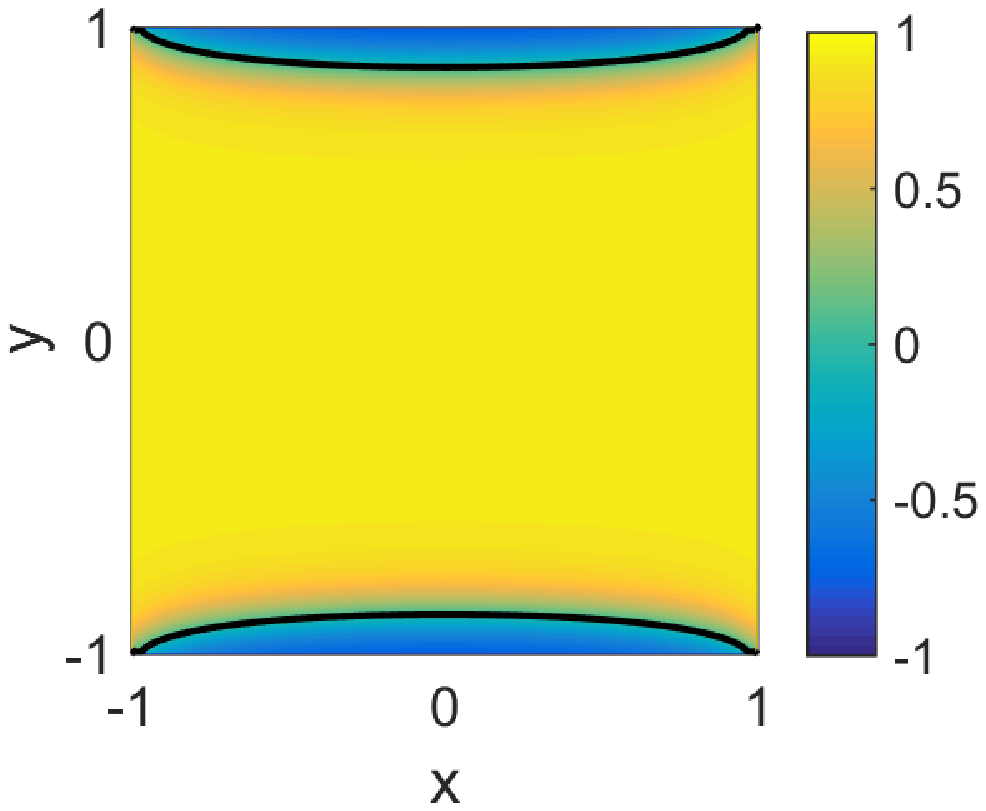}}
	\caption{$q(x,\,y,\,t)$ for $t=0$, $t=0.5$ and $t=2$ for $\lambda^2=200$.}
	\label{fig:2}
\end{figure}

\begin{figure}[h!]
	\centering
 	\centerline{\includegraphics[width=6cm]{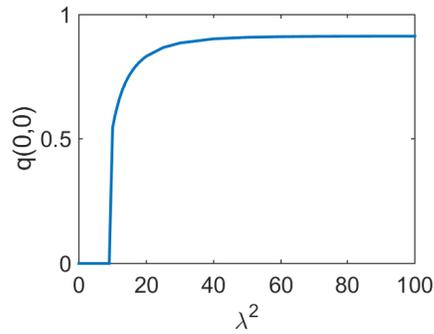}}
	\caption{$q(0,\,0)$ for the steady solution as $\lambda^2$ varies, the critical value is $\lambda^2=9.2$.}
	\label{fig:3}
\end{figure}

\subsection{Order Reconstruction on a Hexagon: Analysis and Numerics}
\label{sec:hexagon}

In this section, we look for OR type solutions on a regular hexagon at the fixed temperature, $A =-\frac{B^2}{3C}$ as before.
We interpret OR solutions loosely i.e. we look for critical points of the LdG energy which have an interior ring of maximal biaxiality 
inside the hexagon. Let~$H$ be a regular hexagon, centered at the origin with vertices $(1, \, 0)$, $(1/2, \, \sqrt{3}/2)$, 
$(-1/2, \, \sqrt{3}/2)$, $(-1, \, 0)$, $(-1/2, \, -\sqrt{3}/2)$, $(1/2, \, -\sqrt{3}/2)$.
We take our domain~$\Omega$ to be the set of points~$(x, \, y)$ in the interior of~$H$ that satisfy the inequalities
\[
 |x| < 1 - \varepsilon, \quad \frac12 \abs{x + \sqrt{3} y} < 1 - \varepsilon,
 \quad \frac12 \abs{x - \sqrt{3} y} < 1 - \varepsilon.
\]
The domain~$\Omega$ is a truncated hexagon (see Figure~\ref{fig:hexagon}) and has the same set of symmetries as the original hexagon~$H$, that is,
\begin{equation} \label{D6}
 \left\{S\in \mathrm{O}(2)\colon S\Omega\subseteq\Omega \right\} = 
 \left\{S\in \mathrm{O}(2)\colon S H\subseteq H \right\} =: D_6.
\end{equation}
The set of symmetries~$D_6$ consists of six reflection symmetries about the symmetry axes of the hexagon, 
and six rotations of angles~$k\pi/3$ for~$k\in\{0, \, \ldots, \, 5\}$.
We label the ``long'' edges of~$\partial\Omega$, that is the edges common with~$\partial H$, as~$C_1, \, \ldots, \, C_6$.
The edges are labelled counterclockwise, starting from~$(1, \, 0)$. 
\begin{figure}[t]
	\centering
 	\includegraphics[height=6cm]{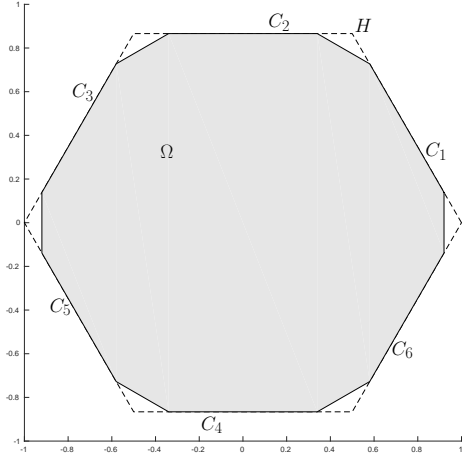}
	\caption{The `truncated hexagon'~$\Omega$. A regular hexagon~$H$ is also plotted, in dashed lines.}
	\label{fig:hexagon}
\end{figure}
On the $C_i$'s, we impose Dirichlet boundary conditions
\begin{equation} \label{eq:bc_hexagon}
 \Qvec(\rvec) = \Qvec_{\mathrm{b}}(\rvec) := 
 \frac{B}{C} \left( \nvec_{\mathrm{b}}(\rvec)\otimes\nvec_{\mathrm{b}}(\rvec) - \frac{\mathbf{I}}{3}\right),
\end{equation}
where~$\nvec_{\mathrm{b}}$ is a tangent unit vector field to~$\partial H$, i.e.
\begin{equation}
\label{eq:hb}
 \nvec_{\mathrm{b}}(\rvec) := \begin{cases}
                        (-1/2, \, \sqrt{3}/2, \, 0)  & \textrm{if } \rvec\in C_1\cup C_4\\
                        (-1, \, 0, \, 0)             & \textrm{if } \rvec\in C_2\cup C_5\\
                        (-1/2, \, -\sqrt{3}/2, \, 0) & \textrm{if } \rvec\in C_3\cup C_6.
                       \end{cases}
\end{equation}
We also impose Dirichlet boundary conditions on the ``short'' edges of~$\partial\Omega$.
For instance, on the short edge connecting the vertices~$(1 - \varepsilon, \, \sqrt{3}\varepsilon)$,
$(1 - \varepsilon, - \sqrt{3}\varepsilon)$, we define
\[
 \Qvec_{\mathrm{b}}(x, \, y) :=
 \frac{B}{C} \left( \nvec_{\mathrm{b}}(x, \, y)\otimes\nvec_{\mathrm{b}}(x, \, y) - \dfrac{\mathbf{I}}{3}\right)
\]
for
\[
 \nvec_{\mathrm{b}}(x, \, y) := \frac{2\varepsilon}{\sqrt{\varepsilon^2 + y^2}}
 \left(-\frac{1}{2}, \, \frac{y}{2\varepsilon}, \, 0\right).
\]
We extend the boundary datum~$\Qvec_{\mathrm{b}}$ to the other short edges by successive rotations of~$\pi/3$.
By construction, the boundary datum is consistent with the symmetries of the hexagon.

We look for critical points of the Landau-de Gennes energy~\eqref{eq:rescaled} on~$\Omega$, such that
(i)~the corresponding $\Qvec$-tensor has $\hat{\mathbf{z}}$ as an eigenvector with constant eigenvalue $-\frac{B}{3C}$, and
(ii) the origin is a uniaxial point with negative scalar order parameter. The long edges are subject to a uniaxial Dirichlet condition 
with positive order parameter (see \eqref{eq:bc_hexagon} and \eqref{eq:hb}) and it is reasonable to expect a ring of maximal biaxiality
separating the central uniaxial point with negative order parameter from the positively ordered uniaxial Dirichlet conditions,
as will be corroborated by the numerics below.

In view of~(i), we look for critical points of the form
\begin{equation} \label{QP}
 \Qvec(\rvec) = \left( \begin{array}{c|c} \mathbf{P}(\rvec) + \dfrac{B}{6C}\mathbf{I}_2 & \begin{matrix} 0 \\ 0 \end{matrix} \\ \hline
                                          \begin{matrix} 0 \qquad & 0 \end{matrix} & -{B}/{3C} \\
                       \end{array} \right),
\end{equation}
where~$\mathbf{P}(\rvec)$ is a $2\times 2$, symmetric and traceless matrix, while~$\mathbf{I}_2$ is the $2\times 2$ identity matrix.
The condition (ii) translates to $\mathbf{P}(0, \, 0) = 0$.
By substitution, we see that~$\Qvec$ is a critical point for the Landau-de Gennes energy~\eqref{eq:rescaled}
if~$\mathbf{P}$ is a solution of the system
\begin{equation} \label{EL_2D}
 \Delta \mathbf{P} = \frac{\lambda^2}{L} \left\{-\frac{B^2}{2C} \mathbf{P} 
 - B \left(\mathbf{P}\mathbf{P} - \frac{\mathbf{I}_2}{2}|\mathbf{P}|^2\right) - C |\mathbf{P}|^2\mathbf{P} \right\}
\end{equation}
or, equivalently, a critical point of the functional
\begin{equation} \label{LdG_2D}
 F[\mathbf{P}] := \int_{\Omega} \left\{ \frac12 \left|\nabla\mathbf{P}\right|^2 
 + \frac{\lambda^2}{L} \left(-\frac{B^2}{4C} \tr\mathbf{P}^2 
 - \frac{B}{3} \tr\mathbf{P}^3 + \frac{C}{4} (\tr\mathbf{P}^2)^2\right) \right\} \, \d A.
\end{equation}
Let~$\mathbf{P}_{\mathrm{b}}$ denote the boundary datum for~$\mathbf{P}$,
which is related to~$\Qvec_{\mathrm{b}}$ via the change of variable~\eqref{QP}.

\begin{lemma} \label{lemma:hexagon}
 For any positive value of~$\lambda$, there exists a critical 
 point~$\mathbf{P}_{\mathrm{s}}\in C^2(\Omega)\cap C^0(\overline\Omega)$ of~\eqref{LdG_2D}
 which satisfy the boundary condition~$\mathbf{P}_{\mathrm{s}} = \mathbf{P}_{\mathrm{b}}$ on~$\partial\Omega$ 
 and~$\mathbf{P}_{\mathrm{s}}(0, \, 0) = 0$.
\end{lemma}

The corresponding~$\Qvec$-tensor, denoted by $\Qvec_{\mathrm{s}}$, is then related to~$\mathbf{P}_{\mathrm{s}}$
via the change of variable~\eqref{QP}, 
and is a critical point of the Landau-de Gennes energy (by construction) with the two desired properties, (i) and (ii) stated above. 

\textit{Proof of Lemma~\ref{lemma:hexagon}.}
 Let~$\mathscr{A}$ be the class of admissible configurations, i.e. maps~$\mathbf{P}\in W^{1, 2}(\Omega, \, S_0^{2\times 2})$ 
 that satisfy the boundary condition~$\mathbf{P} = \mathbf{P}_{\mathrm{b}}$ on~$\partial\Omega$.
 Let~$\mathscr{A}_{\mathrm{sym}}$ be the class of admissible configurations that are consistent with the symmetries of the hexagon i.e.
 $\mathscr{A}_{\mathrm{sym}}$ is the set of maps~$\mathbf{P}\in\mathscr{A}$ that satisfy 
 \begin{equation} \label{symmetry}
  \mathbf{P}(\rvec) = S \mathbf{P}(S^{\mathsf{T}}\rvec) S^{\mathsf{T}}
 \end{equation}
 for a.e.~$\rvec\in\Omega$ and any matrix~$S\in D_6$. Here~$D_6$ is the group of symmetries of the hexagon defined by~\eqref{D6}.
 Recall that the boundary datum~$\mathbf{P}_{\mathrm{b}}$ satisfies~\eqref{symmetry} by construction, 
 so the set~$\mathscr{A}_{\mathrm{sym}}$ is non-empty.
 By a standard application of the Direct Method of the Calculus of Variations, we can prove the existence of a
 minimiser~$\mathbf{P}_{\mathrm{s}}$ for the energy~$F$ given by~\eqref{LdG_2D},
 in the class~$\mathscr{A}_{\mathrm{sym}}$.
 
 Clearly,~$\mathbf{P}_{\mathrm{s}}$ is a critical point for~$F$ restricted to~$\mathscr{A}_{\mathrm{sym}}$, 
 but we do not know if it is a critical point for~$F$ in~$\mathscr A$, hence a solution of the Euler-Lagrange system~\eqref{EL_2D}.
 However, the right-hand side of formula~\eqref{symmetry} defines an isometric action of the group~$D_6$ 
 on the space of admissible maps~$\mathscr{A}$,
 and the energy~$F$ is invariant with respect to this action. Therefore, we can apply Palais' principle of symmetric
 criticality~\cite[Theorem p.~23]{palais} to conclude that critical points of $F$ in the restricted space $\mathscr{A}_{\mathrm{sym}}$ exist as critical points in the space $\mathscr{A}$.
 We conclude that~$\mathbf{P}_{\mathrm{s}}$ is a critical point of~$F$ in~$\mathscr{A}$,
 i.e. a solution of~\eqref{EL_2D}.
 By elliptic regularity, we obtain that~$\mathbf{P}_{\mathrm{s}}\in C^2(\Omega)\cap C^0(\overline\Omega)$.
 Finally, we evaluate~\eqref{symmetry} at the point~$\rvec = (0, \, 0)$  to obtain that
 \[
  \mathbf{P}(0, \, 0) = S \mathbf{P}(0, \, 0) S^{\mathsf{T}} \qquad \textrm{for any } S\in D_6,
 \]
which necessarily requires that~$\mathbf{P}_{\mathrm{s}}(0, \, 0) = 0$ as stated.
\endproof

Next, we perform numerical experiments on a regular hexagon of edge length $\lambda$, with the gradient flow model \eqref{eq:e2}, 
to investigate the stability of the OR-type critical point constructed in Lemma~\ref{lemma:hexagon}. 
We re-scale the spatial coordinates by $\bar{\rvec}_i = \frac{\rvec_i}{\lambda}$ as above and solve the system of five coupled 
partial differential equations for $Q_{ij}$, $i,j=1,2,3$ with different values of $\bar{\lambda}^2 :=\frac{\lambda^2}{L}$ 
at $A = -\frac{B^2}{3C}$.

We impose Dirichlet conditions on all six edges of the form
\begin{equation}
\label{eq:hb2}
\Qvec_{\mathrm{b}}(x, \, y) = \frac{B}{C}\left(\nvec_{\mathrm{b}} \otimes \nvec_{\mathrm{b}} - \frac{\mathbf{I}}{3} \right)
\end{equation}
and there are discontinuities at the vertices.
The choice of $\nvec_{\mathrm{b}}$ is dictated by the tangent unit-vector to the edge in question i.e. see \eqref{eq:hb},
and at a given vertex, we fix $\Qvec_{\mathrm{b}}$ to be the average of the two intersecting edges.

We impose an initial condition which divides the hexagon into six regions, which are three alternating constant uniaxial states,
as demonstrated in Figure~\ref{fig:initial}. This initial condition is not well defined at the origin but this does not pose to be
a problem for the numerics. We look for solutions which have $\zhat$ as an eigenvector and have a uniaxial point at the origin with
negative order parameter. This translates to (i)~$Q_{33} = -\frac{B}{3C}$ everywhere, (ii) $Q_{13} = Q_{23}=0$ everywhere,
(ii) $Q_{11}=Q_{22}=\frac{B}{6C}$ at the origin and (iii) $Q_{12}=0$ at the origin. 

\begin{figure}[h!]
	\centering
 	\centerline{\includegraphics[width=5cm]{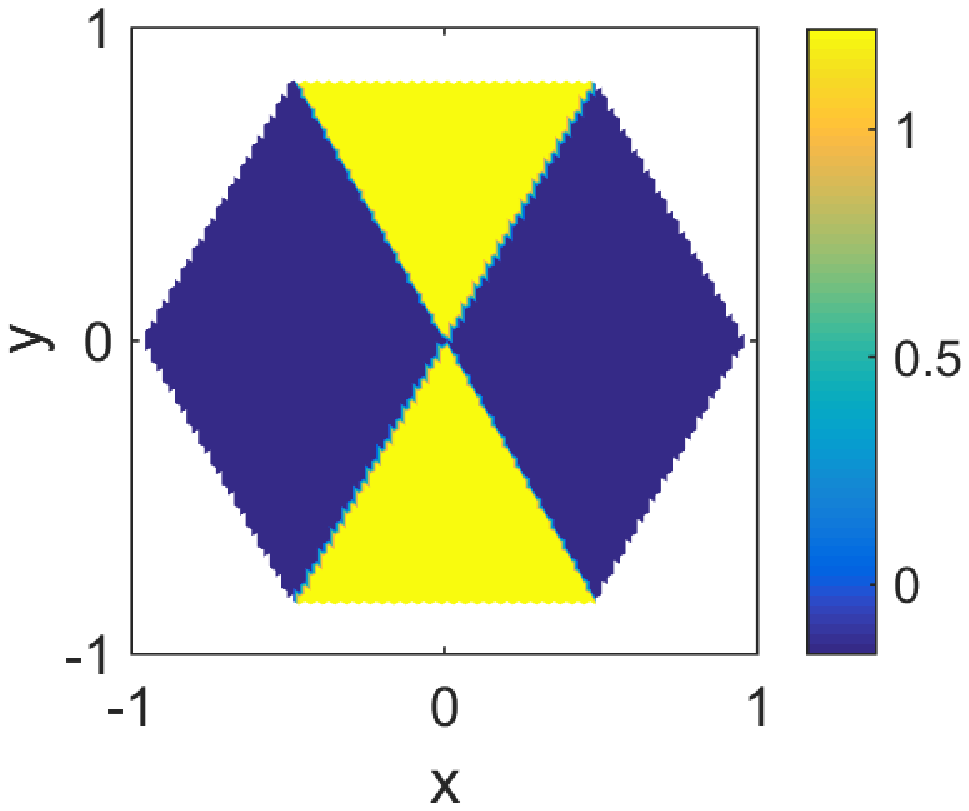}\includegraphics[width=5cm]{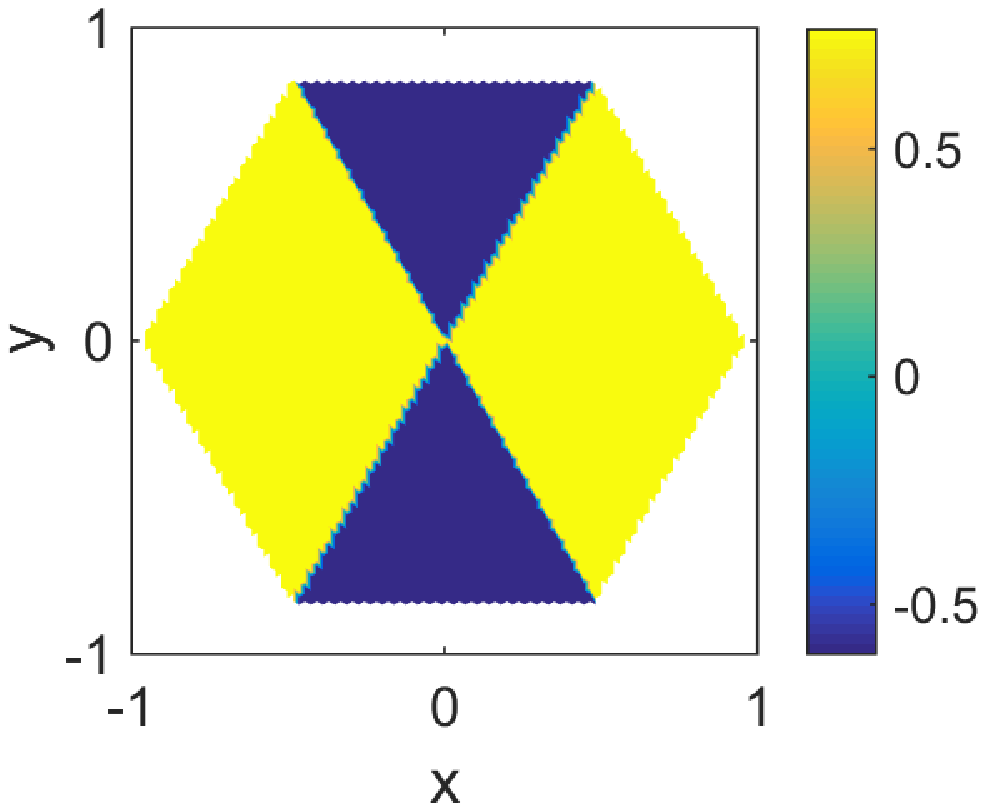}\includegraphics[width=5cm]{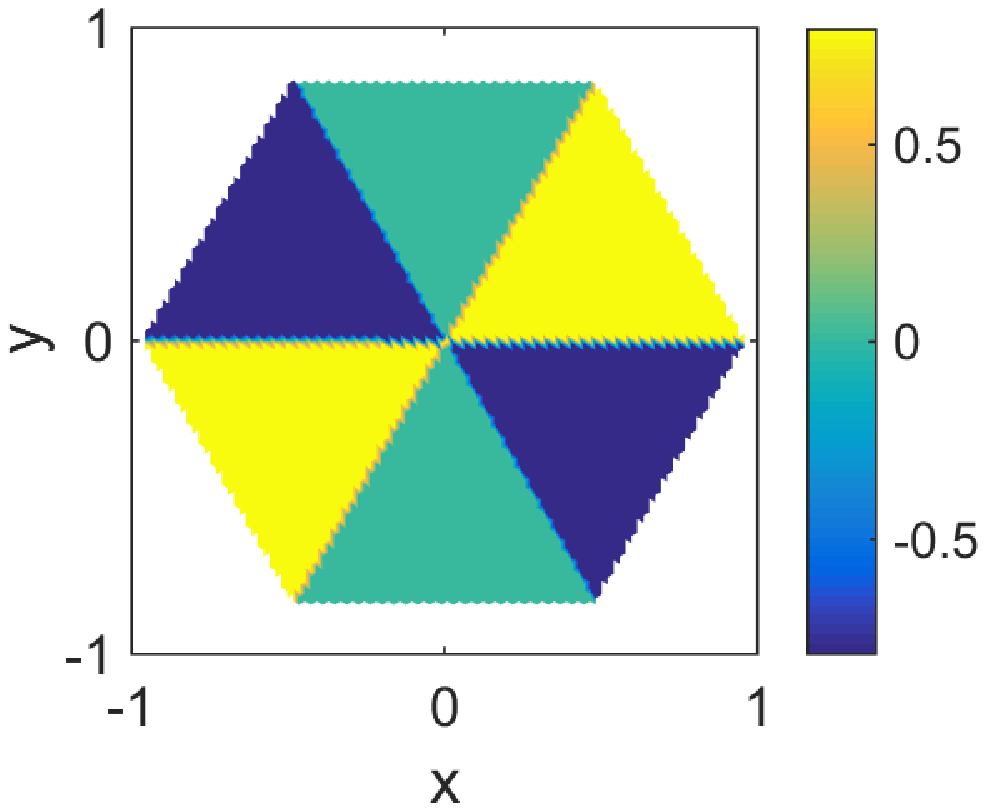}}
	\caption{$Q_{11}$, $Q_{22}$ and $Q_{12}$ for $t=0$.}
	\label{fig:initial}
\end{figure}

We solve the gradient-flow system on $H$ with $\bar{\lambda}^2 = 10^{-6}$, with the fixed Dirichlet condition and
initial condition as described above.
In Figures~\ref{fig:h1},\ref{fig:h2}, we plot $Q_{12}, Q_{11}, Q_{22}$ of the converged solution and see that the origin
is indeed an uniaxial point with negative scalar order parameter i.e. the dynamic solution at the origin is given by
\begin{equation}
 \Qvec(0, \, 0, \, t) = -\frac{B}{2C}\left(\zhat\otimes\zhat - \frac{\mathbf{I}}{3} \right)
\end{equation}
for large times. Further, we numerically verify that
\begin{equation}
\label{eq:Q33}
 Q_{33}(\rvec, \, t)= -\frac{B}{3C}, \qquad Q_{13}(\rvec, \, t) =Q_{23}(\rvec, \, t) = 0 
\end{equation}
for all times~$t$, so that $\zhat$ is indeed an eigenvector with constant eigenvalue.

In Figure~\ref{fig:h3}, we plot the biaxiality parameter, $\beta^2$, of the converged solution
\[
 \beta^2 = 1 - \frac{\left(\textrm{tr}\Qvec^3\right)^2}{|\Qvec|^6} \in \left[0, 1 \right]
\]
at $\bar{\lambda}^2 = 10^{-6}$ and see a distinct ring of maximal biaxiality (with $\beta^2=1$ 
such that~$\Qvec$ has a zero eigenvalue) around the origin, hence yielding an OR-type solution on a regular hexagon.

In Figures~\ref{fig:h4}, \ref{fig:h5}, \ref{fig:h6}, we plot the components of the converged solutions at the origin
as a function of $\bar{\lambda}^2$. We see that \eqref{eq:Q33} holds for all $\lambda^2$ so that $\zhat$ is always an eigenvector. 
Further, the converged solution respects $Q_{12}=0, Q_{11}=Q_{22}=\frac{B}{6C}$ at the origin for $\bar{\lambda}^2 \leq 0.002$
and hence, we have an uniaxial point with negative order parameter at the origin for $\bar{\lambda}^2 \leq 0.002$. 
The numerics suggest that we have an OR-type solution on a regular hexagon for $\bar{\lambda}^2 \leq 0.002$, 
which loses stability for larger values of $\bar{\lambda}^2$. The qualitative trends are the same as those observed on a regular square.

We can compare the critical value on a hexagon with that obtained on a square, see \eqref{eq:e10}.
Our numerics suggest that an OR-type solution is locally stable on a regular hexagon of edge length $\lambda$ for
\begin{equation}
\label{hs}
\lambda^2 < 7 \frac{L}{C}.
\end{equation}

\begin{figure}[h!]
	\centering
 	\includegraphics[width=5cm]{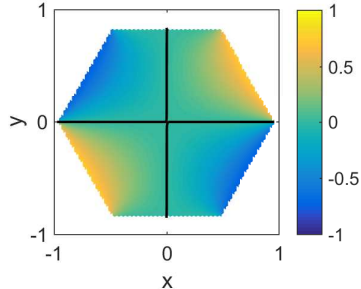}\\ 
	\caption{$Q_{12}$ with contours at level 0 for $\lambda^2=10^{-6}$ and $t=2$.}
	\label{fig:h1}
\end{figure}

\begin{figure}[h!]
	\centering
 	\includegraphics[width=5cm]{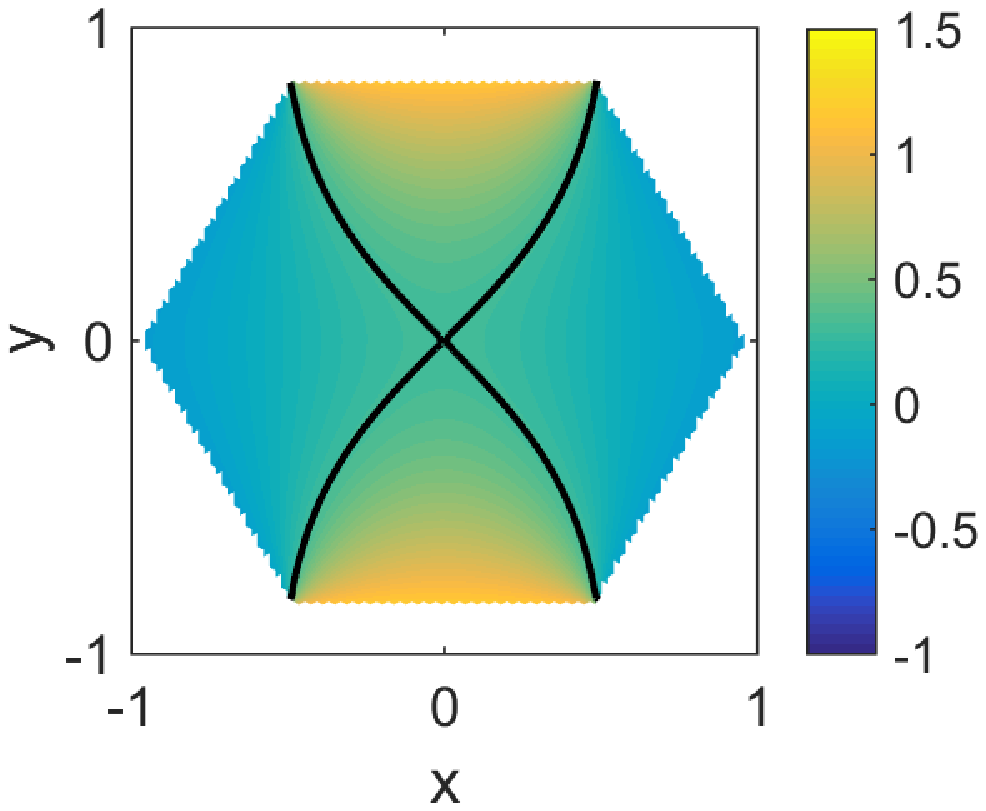}\includegraphics[width=5cm]{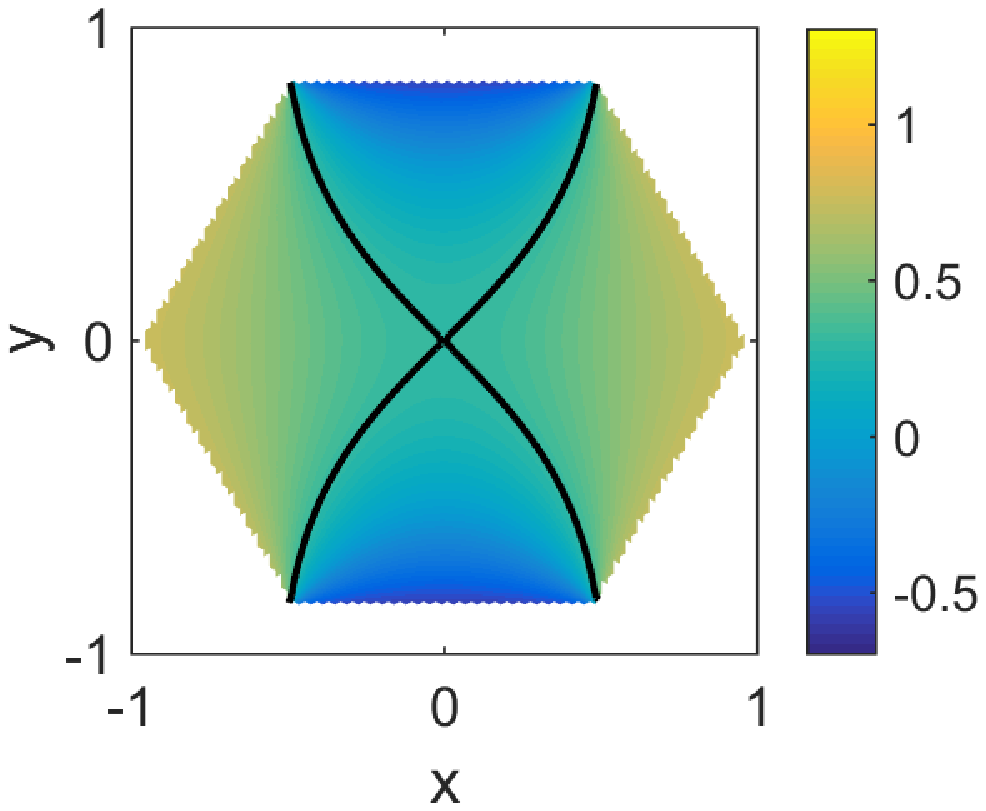}
	\caption{$Q_{11}$ and $Q_{22}$ with contours at level $\frac{B}{6C}$ for $\lambda^2=10^{-6}$ and $t=2$.}
	\label{fig:h2}
\end{figure}

\begin{figure}[h!]
	\centering
 	\includegraphics[width=5.2cm]{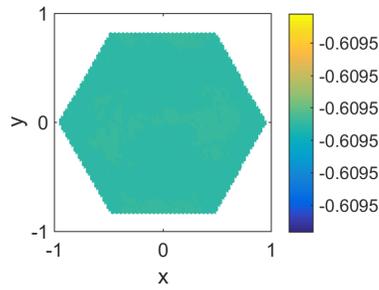}
	\caption{$Q_{33}=-Q_{11}-Q_{22}$ for $\lambda^2=10^{-6}$ and $t=2$.}
	\label{fig:h3}
\end{figure}

\begin{figure}[h!]
	\centering
 	\includegraphics[width=5cm]{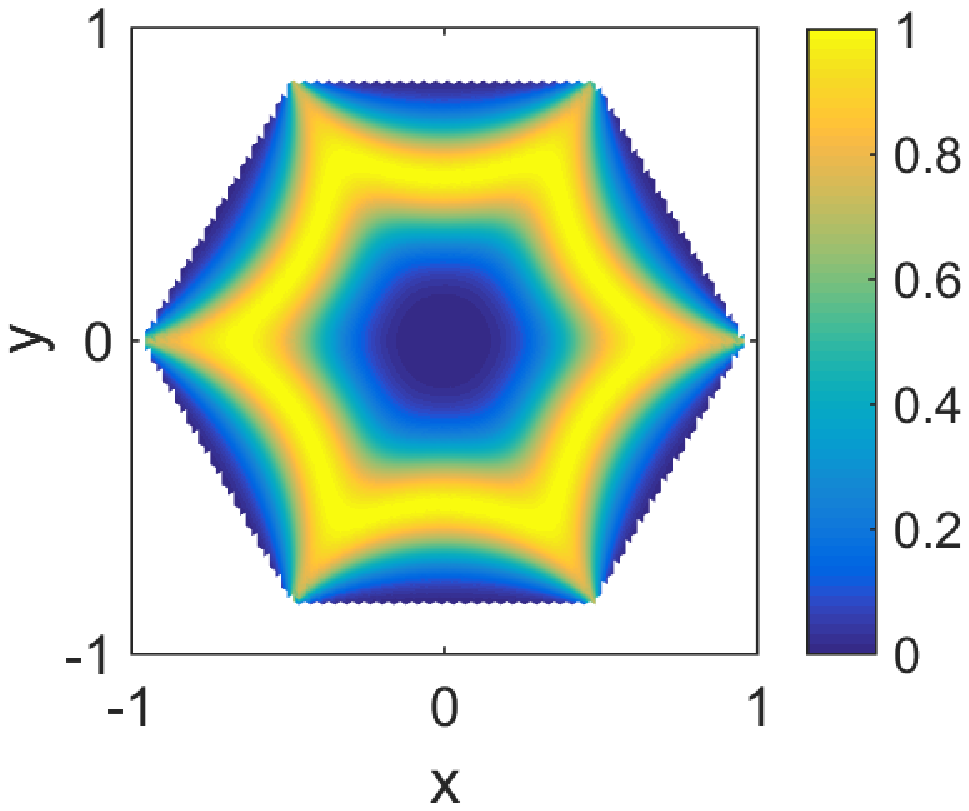}\includegraphics[width=5cm]{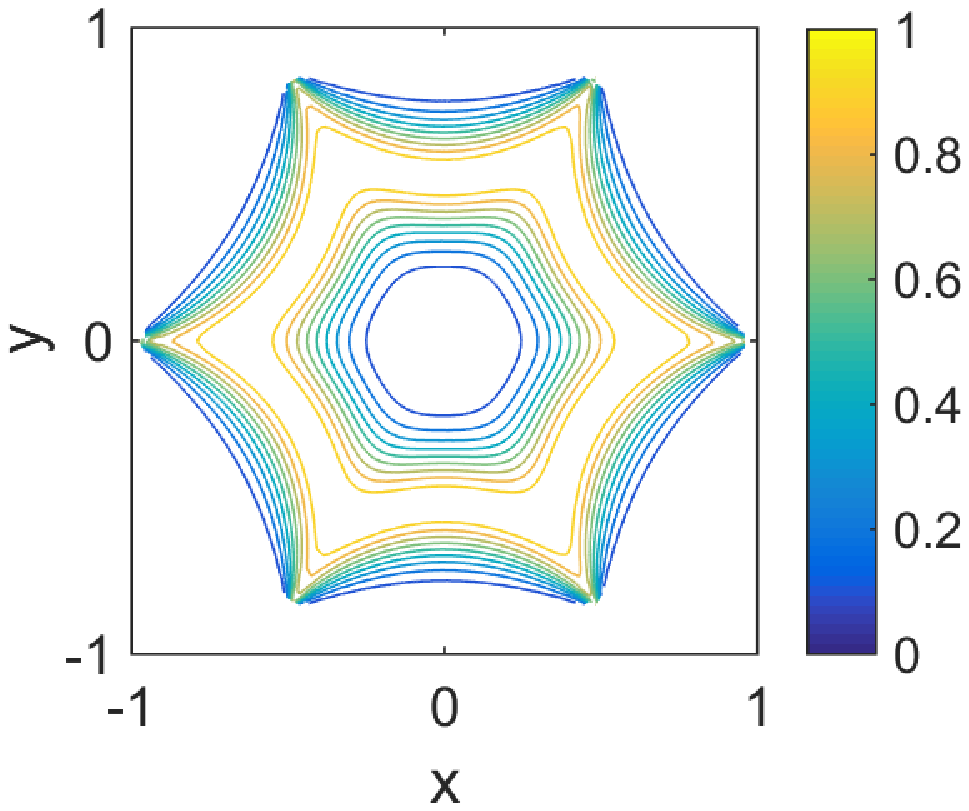}
	\caption{Plot and contour plot of biaxiality parameter $\beta^2$ for $\lambda^2=10^{-6}$ and $t=2$.}
\end{figure}

\begin{figure}[h!]
	\centering
	\includegraphics[width=6cm]{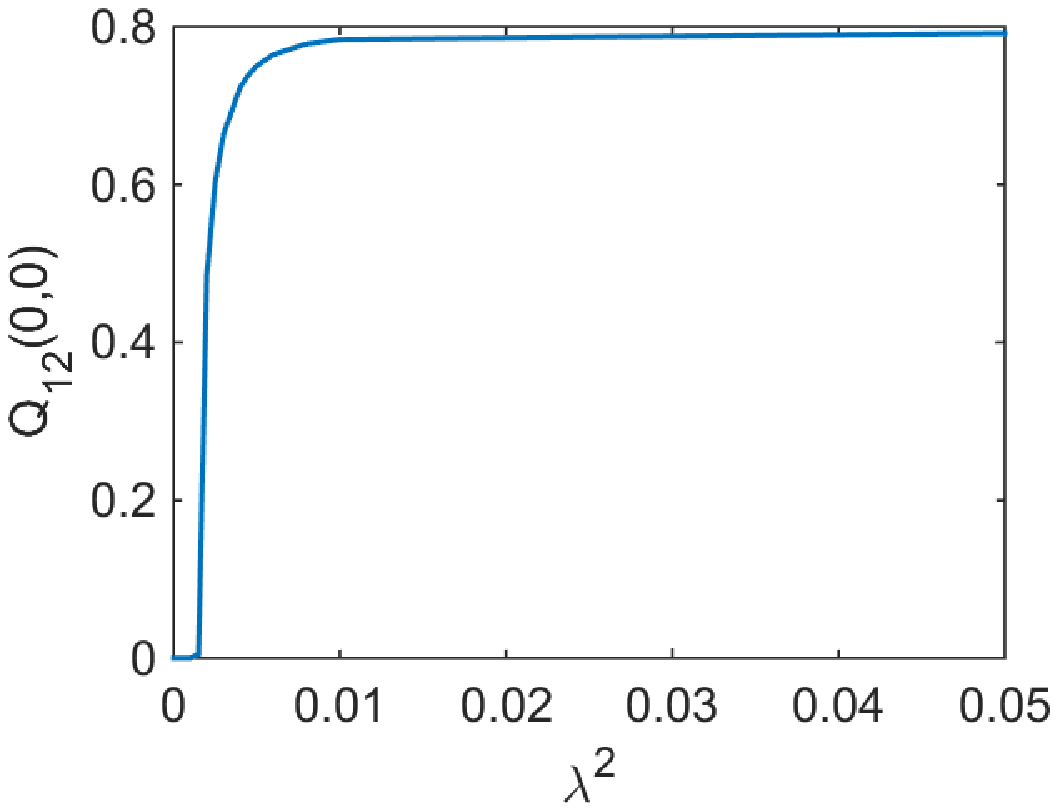}
	\caption{$Q_{12}$ at the origin for various $\lambda$. The critical value of $\lambda^2=0.002$.}
	\label{fig:h4}
\end{figure}

\begin{figure}[h!]
	\centering
 	\includegraphics[width=6cm]{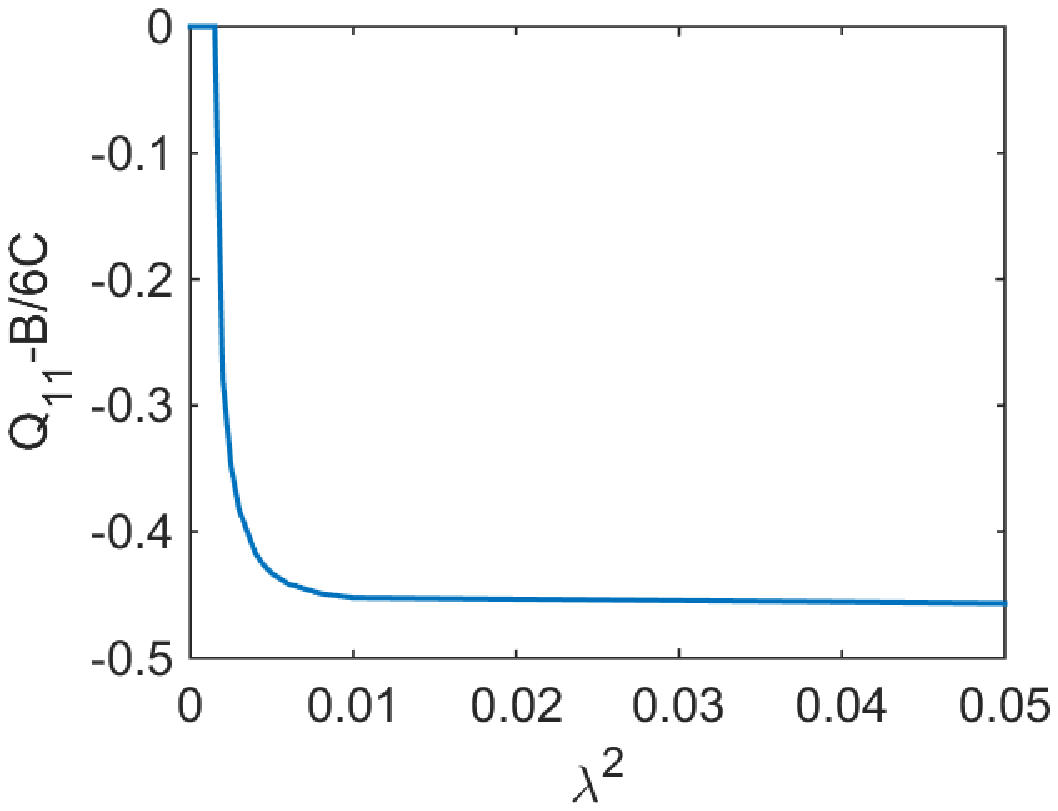}
	\caption{$Q_{11}-\frac{B}{6C}$ at the origin for various $\lambda$. The critical value of $\lambda^2=0.002$.}
	\label{fig:h5}
\end{figure}

\begin{figure}[h!]
	\centering
 	\includegraphics[width=6cm]{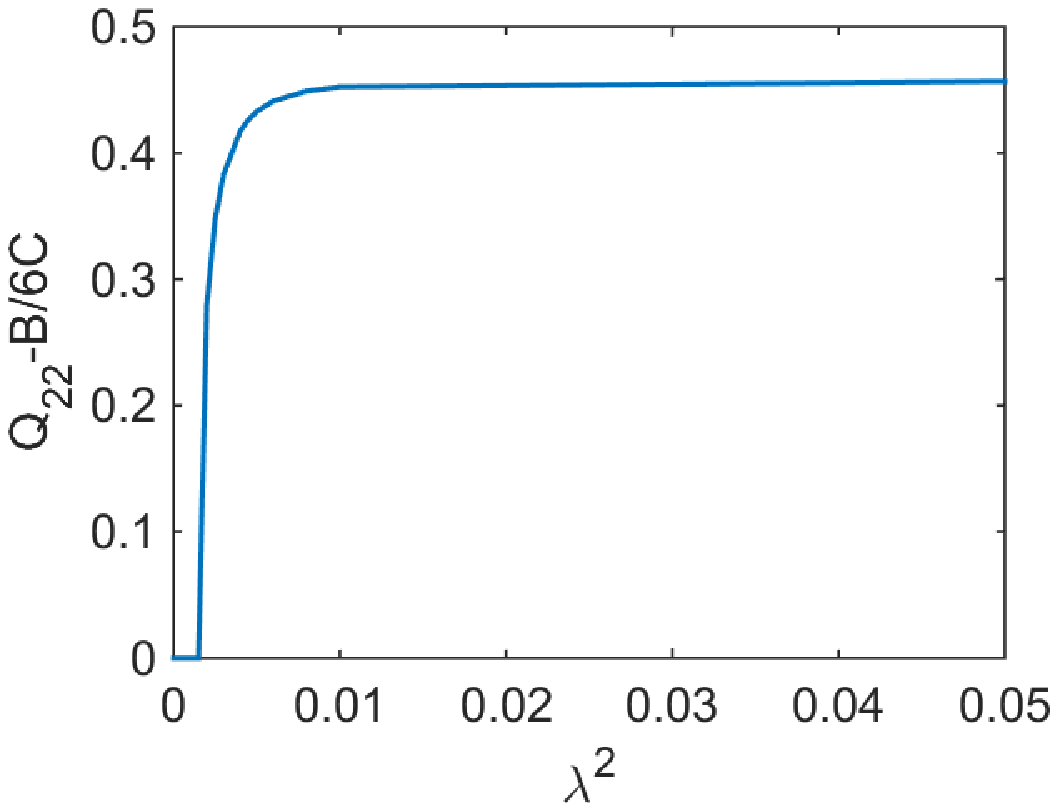}
	\caption{$Q_{22}-\frac{B}{6C}$ at the origin for various $\lambda$. The critical value of $\lambda^2=0.002$.}
	\label{fig:h6}
\end{figure}

\section{Conclusions}
\label{sec:conclusions}

We analytically and numerically study an OR-type LdG critical point on a square domain at a fixed temperature, motivated by the WORS critical point reported in \cite{kraljmajumdar}
distinguished by a uniaxial cross with negative order parameter along the square diagonals. The OR-type critical point is globally stable for edge lengths comparable to the biaxial
correlation length of the order of $\sqrt{L}{C}$. The convergence result in Proposition~\ref{prop:3} gives insight into how the diagonal cross deforms
into uniaxial transition layers of negative order parameter, along the square edges. Recent numerical experiments show that there is a continuous branch of critical points emanating 
from the OR critical point \cite{robinson} for which the uniaxial cross continuously deforms from the diagonal towards the edges; in some cases, there can be up to $81$ critical points 
for a given $\lambda$. Further, our preliminary numerical investigations on a square and a hexagon suggest that OR-type critical points are exist and are globally stable for regular
two-dimensional polygons with an even number of sides, when the side length is sufficiently small, and the OR critical points lose stability by undergoing a supercritical pitchfork 
bifurcation as the edge length increases. We will study the generic character of OR-type critical points in future work.

\section{Acknowledgements}

G.C. has received funding from the European Research Council under the European Union's Seventh Framework Programme (FP7/2007-2013) / ERC grant agreement n° 291053.
A.M. is supported by an EPSRC Career Acceleration Fellowship EP/J001686/1 and EP/J001686/2 and an OCIAM Visiting Fellowship.
A.S. is supported by an Engineering and Physical Sciences Research Council (EPSRC) studentship.

\end{document}